\documentclass[reqno]{amsart}
\usepackage{amsthm,amsmath,amsmath,amssymb,amsbsy,graphics,graphicx,hyperref}
\newcommand{\comment}[1]{{}}

\theoremstyle{plain}
\newtheorem{theorem}{Theorem}
\numberwithin{theorem}{section}
\newtheorem{lemma}[theorem]{Lemma}
\newtheorem{proposition}[theorem]{Proposition}
\newtheorem{prop-def}[theorem]{Proposition-Definition}
\newtheorem{corollary}[theorem]{Corollary}

\newcounter{savetheorem} 
\newcounter{savesection}
\newcounter{intro-section-counter}
\newcounter{SMTT-counter}
\newcounter{WSMTT-counter}
\newcounter{shifted-counter}
\newcounter{zpoly-counter}

\theoremstyle{definition}
\newtheorem{definition}[theorem]{Definition}

\newtheorem{example}[theorem]{Example}

\theoremstyle{remark}
\newtheorem{remark}[theorem]{Remark}

\newcommand{\bd}{\partial} 
\newcommand{\cbd}{\partial^*}
\newcommand{\wbd}{\hat\partial} 
\newcommand{\wcbd}{\hat\partial^*}
\newcommand{\Abd}{\boldsymbol{\partial}}  
\newcommand{\Acbd}{\boldsymbol{\partial}^*}

\newcommand{\stot}{s^{\rm tot}} 
\newcommand{\sud}{s^{\rm ud}} 
\newcommand{\sdu}{s^{\rm du}} 
\newcommand{\Ltot}{L^{\rm tot}} 
\newcommand{\Lud}{L^{\rm ud}} 
\newcommand{\Ldu}{L^{\rm du}} 

\newcommand{\Swud}{\hat{s}^{ud}} 
\newcommand{\Lwud}{\hat{L}^{\rm ud}} 

\newcommand{\Stot}{\mathbf{s}^{tot}} 
\newcommand{\Sud}{\mathbf{s}^{ud}} 
\newcommand{\Sdu}{\mathbf{s}^{du}} 
\newcommand{\LLtot}{\mathbf{L}^{\rm tot}} 
\newcommand{\LLud}{\mathbf{L}^{\rm ud}} 
\newcommand{\LLdu}{\mathbf{L}^{\rm du}} 
\newcommand{\LL}{\mathbf{L}} 

\newcommand{\del}{\mathop{\rm del}\nolimits}
\newcommand{\id}{\mathop{\rm id}\nolimits}
\newcommand{\im}{\mathop{\rm im}\nolimits}
\newcommand{\link}{\mathop{\rm link}\nolimits}

\newcommand{\rank}{\mathop{\rm rank}\nolimits}

\newcommand{\Z}{\mathbb{Z}}
\newcommand{\SST}{{\mathcal T}}
\newcommand{\HH}{\tilde H} 
\newcommand{\Cc}{\mathbb{C}}
\newcommand{\Nn}{\mathbb{N}}
\newcommand{\Qq}{\mathbb{Q}}
\newcommand{\Zz}{\mathbb{Z}}

\newcommand{\abs}[1]{\lvert #1 \rvert}
\newcommand{\cpless}{\prec} 
\newcommand{\cpleq}{\preceq} 
\newcommand{\cpmoreeq}{\succeq} 
\newcommand{\dju}{\,\dot{\cup}\,} 
\newcommand{\betti}{\tilde\beta} 
\newcommand{\zeq}{\overset{\circ}{=}}
\newcommand{\st}{\colon}
\newcommand{\0}{\emptyset}
\newcommand{\shgen}[1]{{\langle{#1}\rangle}}  

\newcommand{\fld}{\Bbbk}
\newcommand{\sign}{\varepsilon}
\newcommand{\tensor}{\otimes}
\newcommand{\dsum}{\oplus}
\newcommand{\x}{\times}
\newcommand{\isom}{\cong}
\newcommand{\sm}{\backslash}

\newcommand{\Up}{\uparrow\!\!}

\newcommand{\sg}{\sigma}           
\newcommand{\setsg}{\sigma}        
\newcommand{\lsg}{\bar{\sigma}}    
\newcommand{\setlsg}{\bar{\sigma}} 
\newcommand{\family}{\mathcal F}   

\begin{document}
\date{August 14, 2008}
\title{Simplicial Matrix-Tree Theorems}

\author{Art M.\ Duval}
\address{Department of Mathematical Sciences,
University of Texas at El Paso,
El Paso, TX 79968-0514}
\author{Caroline J.\ Klivans}
\address{Departments of Mathematics and Computer Science,
The University of Chicago,
Chicago, IL 60637}
\author{Jeremy L.\ Martin}
\address{Department of Mathematics,
University of Kansas,
Lawrence, KS 66047}

\thanks{Second author partially supported by NSF VIGRE grant DMS-0502215.
Third author partially supported by an NSA Young Investigators Grant.}


\begin{abstract}
We generalize the definition and enumeration of spanning trees
from the setting of graphs to that of arbitrary-dimensional simplicial
complexes $\Delta$, extending an idea due to G.~Kalai.
We prove a simplicial version of the Matrix-Tree Theorem that
counts simplicial spanning trees,
weighted by the squares of the orders of their top-dimensional integral
homology groups, in terms of the Laplacian matrix of~$\Delta$.
As in the graphic case, one can obtain a more finely weighted
generating function for simplicial spanning trees by assigning
an indeterminate to each vertex of $\Delta$
and replacing the entries of the Laplacian with Laurent monomials.
When $\Delta$ is a shifted complex, we give a combinatorial interpretation
of the eigenvalues of its weighted Laplacian and prove that they determine
its set of faces uniquely, generalizing known results about
threshold graphs and unweighted Laplacian eigenvalues of shifted complexes.
\end{abstract}

\maketitle

\section{Introduction} \label{intro-section}
\setcounter{intro-section-counter}{\value{section}}

This article is about generalizing the Matrix-Tree Theorem
from graphs to simplicial complexes.

\subsection{The classical Matrix-Tree Theorem}
We begin by reviewing the classical case; for a more detailed treatment,
see, e.g., \cite{Bollobas}.
Let $G$ be a finite, simple, undirected graph with vertices
$V(G)=[n]=\{1,2,\dots,n\}$ and edges $E(G)$.
A \emph{spanning subgraph} of $G$ is a graph $T$ with $V(T)=V(G)$
and $E(T)\subseteq E(G)$; thus a spanning subgraph may be specified
by its edge set.
A spanning subgraph~$T$ is a \emph{spanning tree} if
(a) $T$ is acyclic; (b) $T$ is connected; and (c) $|E(T)|=|V(T)|-1$.
It is a fundamental property of spanning trees (the ``two-out-of-three theorem'')
that any two of these three conditions together imply the third.

The \emph{Laplacian} of $G$ is the $n\x n$ symmetric matrix $L=L(G)$ with entries
  $$L_{ij} = \begin{cases}
    \deg_G(i) & \text{ if } i=j,\\
    -1 & \text{ if $i,j$ are adjacent,}\\
    0 & \text{ otherwise,}
    \end{cases}$$
where $\deg_G(i)$ is the degree of vertex~$i$ (the number of edges having $i$
as an endpoint).  Equivalently, $L=\bd\cbd$, where $\bd$ is the (signed) vertex-edge
incidence matrix and $\cbd$ is its transpose.  If we regard $G$ as a one-dimensional simplicial
complex, then $\bd$ is just the simplicial boundary map from 1-faces
to 0-faces, and $\cbd$ is the simplicial coboundary map.  The matrix $L$ is
symmetric, hence diagonalizable, so it has $n$ real eigenvalues (counting multiplicities).
The number of nonzero eigenvalues of $L$ is $n-c$, where $c$
is the number of components of $G$.

The Matrix-Tree Theorem, first observed by Kirchhoff \cite{Kirchhoff}
in his work on electrical circuits (modern references include \cite{Bollobas},
\cite{Moon} and \cite[Chapter~5]{EC2}), expresses the number $\tau(G)$ of spanning trees of $G$
in terms of $L$.  The theorem has two equivalent formulations.

\begin{theorem}[{\bf Classical Matrix-Tree Theorem}]
\label{CMTT}
Let $G$ be a connected graph with $n$ vertices, and let $L$ be its Laplacian matrix.
\begin{enumerate}
\item If the eigenvalues of $L$ are $\lambda_0=0,\lambda_1,\dots,\lambda_{n-1}$, then
  $$\tau(G) = \frac{\lambda_1\cdots\lambda_{n-1}}{n}.$$
\item For $1\leq i\leq n$, let $L_i$ be the \emph{reduced Laplacian} obtained from
$L$ by deleting the $i^{th}$ row and $i^{th}$ column.  Then
  $$\tau(G) = \det L_i.$$
\end{enumerate}
\end{theorem}

Well-known corollaries of the Matrix-Tree Theorem include Cayley's formula \cite{Cayley}
  \begin{equation} \label{Cayley}
  \tau(K_n)=n^{n-2}
  \end{equation}
where $K_n$ is the complete graph on $n$ vertices, and Fiedler and Sedl\'{a}\c{c}ek's
formula \cite{FS}
  \begin{equation} \label{Kmn}
  \tau(K_{n,m})=n^{m-1}m^{n-1}
  \end{equation}
where $K_{n,m}$ is the complete bipartite graph on vertex sets of sizes $n$ and $m$.

The Matrix-Tree Theorem can be refined by introducing an indeterminate $e_{ij}=e_{ji}$
for each pair of vertices $i,j$, setting $e_{ij}=0$ if $i,j$ do not share a common edge.
The \emph{weighted Laplacian} $\hat L$
is then defined as the $n\x n$ matrix with entries
  $$\hat L_{ij} = \begin{cases}
    \sum_{k=1}^n e_{ik} & \text{ if } i=j,\\
    -e_{ij} & \text{ if $i,j$ are adjacent,}\\
    0 & \text{ otherwise.}
  \end{cases}$$

\begin{theorem}[{\bf Weighted Matrix-Tree Theorem}]
Let $G$ be a graph with $n$ vertices, and let $\hat L$ be its weighted Laplacian matrix.
\begin{enumerate}
\item If the eigenvalues of $\hat L$ are $\hat\lambda_0=0,\hat\lambda_1,\dots,\hat\lambda_{n-1}$, then
  $$\sum_{T\in\SST(G)} \prod_{ij\in T} e_{ij} =
      \frac{\hat\lambda_1\cdots\hat\lambda_{n-1}}{n}$$
where $\SST(G)$ is the set of all spanning trees of $G$.  
\item For $1\leq i\leq n$, let $\hat L_i$ be the \emph{reduced weighted Laplacian} obtained from
$\hat L$ by deleting the $i^{th}$ row and $i^{th}$ column.  Then
    $$\sum_{T\in\SST(G)} \prod_{ij\in T} e_{ij} = \det\hat L_i.$$
\end{enumerate}
\end{theorem}

By making appropriate substitutions for the indeterminates $e_{ij}$,
it is often possible to obtain finer enumerative information than merely the number of
spanning trees.
For instance, when $G=K_n$, introducing indeterminates $x_1,\dots,x_n$
and setting $e_{ij}=x_ix_j$ for all $i,j$ yields the \emph{Cayley-Pr\"ufer Theorem},
which enumerates spanning trees of $K_n$ by their degree sequences:
  \begin{equation} \label{Cayley-Prufer}
  \sum_{T\in\SST(G)} \ \prod_{i=1}^n x_i^{\deg_T(i)} = x_1\cdots x_n(x_1+\cdots+x_n)^{n-2}.
  \end{equation}
Note that Cayley's formula \eqref{Cayley} can be
recovered from the Cayley-Pr\"ufer Theorem by setting $x_1=\cdots=x_n=1$.

\subsection{Simplicial spanning trees and how to count them}

To extend the scope of the Matrix-Tree Theorem
from graphs to simplicial complexes, we must first say what
``spanning tree'' means in arbitrary dimension.  Kalai \cite{Kalai}
proposed a definition that replaces the
acyclicity, connectedness, and edge-count conditions with
their analogues in simplicial homology.
Our definition adapts Kalai's definition to a more general class
of simplicial complexes.%
  \footnote{There are many other definitions of ``simplicial tree'' in the literature,
  depending on which properties of trees one wishes to extend;
  see, e.g., \cite{BP,Dewdney,Faridi,HLM,MV}.  By adopting Kalai's idea, we choose a
  definition that lends itself well to enumeration.  The closest to ours in spirit
  is perhaps that of Masbaum and Vaintrob \cite{MV}, whose main result is a Matrix-Tree-like
  theorem enumerating a different kind of 2-dimensional tree using Pfaffians rather than Laplacians.}

Let $\Delta$ be a $d$-dimensional simplicial complex, and let $\Upsilon\subset\Delta$
be a subcomplex containing all faces of $\Delta$ of dimension $<d$.  We say that
$\Upsilon$ is a \emph{simplicial spanning tree} of $\Delta$ if the following
three conditions hold:
 \begin{subequations}
  \begin{align}
  & \HH_d(\Upsilon,\Zz) = 0,\label{ST-acyclic}\\
  & |\HH_{d-1}(\Upsilon,\Zz)| < \infty, \text{ and }\label{ST-conn}\\
  & f_d(\Upsilon) = f_d(\Delta)-\betti_d(\Delta)+\betti_{d-1}(\Delta),\label{ST-count}
  \end{align}
  \end{subequations}
where $\HH_i$ denotes reduced simplicial homology (for which see, e.g.,
\cite[\S2.1]{Hatcher}).  (The conditions \eqref{ST-acyclic} and \eqref{ST-conn}
were introduced by Kalai in \cite{Kalai}, while \eqref{ST-count} is more general,
as we will explain shortly.)  When $d=1$, the conditions
\eqref{ST-acyclic}\dots\eqref{ST-count} say respectively that $\Upsilon$ is acyclic, connected,
and has one fewer edge than it has vertices, recovering the definition of the
spanning tree of a graph.  Moreover, as we will show in
Proposition \ref{two-out-of-three}, any two of the three conditions
together imply the third.

A graph $G$ has a spanning tree if and only if $G$ is connected.
The corresponding condition for a simplicial complex $\Delta$
of dimension~$d$ is that $\HH_i(\Delta,\Qq)=0$ for all $i<d$;
that is, $\Delta$ has the rational homology type of a wedge of $d$-dimensional
spheres.  We will call such a complex \emph{acyclic in positive codimension}, or APC for short.
This condition, which we will assume throughout the rest of the introduction, is much
weaker than Cohen-Macaulayness (by Reisner's theorem \cite{Reisner}),
and therefore encompasses many complexes of combinatorial interest,
including all connected graphs, simplicial spheres, shifted, matroid, and
Ferrers complexes, and some chessboard and matching complexes.

For $k\leq d$, let $\bd=\bd_k$ be the $k$th simplicial boundary
matrix of~$\Delta$ (with rows and columns indexed respectively by $(k-1)$-dimensional
and $k$-dimensional faces of $\Delta$), and let $\cbd$ be its transpose.
The \emph{($k$th up-down) Laplacian} of $\Delta$ is $L=\bd\cbd$; this can be regarded either as a square matrix of size $f_{k-1}(\Delta)$ or as a linear endomorphism on ($k-1$)-chains of $\Delta$.  Define invariants
  \begin{align*}
  \pi_k = \pi_k(\Delta) &= \text{ product of all nonzero eigenvalues of $L$},\\
  \tau_k = \tau_k(\Delta) &= \sum_{\Upsilon\in\SST_k(\Delta)} |\HH_{k-1}(\Upsilon)|^2,
  \end{align*}
where $\SST_k(\Delta)$ denotes the set of all $k$-trees of $\Delta$ (that is,
simplicial spanning trees of the $k$-skeleton of $\Delta$).

Kalai \cite{Kalai} studied these invariants in the case that $\Delta$ is
a simplex on $n$ vertices, and proved the formula
  \begin{equation} \label{Kalai-Cayley}
  \tau_k(\Delta) = n^{\binom{n-2}{k}}
  \end{equation}
(of which Cayley's formula \eqref{Cayley} is the special case $k=1$).  Kalai also
proved a natural weighted analogue of \eqref{Kalai-Cayley} enumerating simplicial
spanning trees by their degree sequences, thus generalizing the Cayley-Pr\"ufer
Theorem~\eqref{Cayley-Prufer}.

Given disjoint vertex sets $V_1,\dots,V_r$ (``color classes''),
the faces of the corresponding \emph{complete colorful complex} $\Gamma$ are those
sets of vertices with no more than one vertex of each color.  Equivalently,
$\Gamma$ is the simplicial join $V_1*V_2*\cdots*V_r$ of the 0-dimensional complexes $V_i$.
Adin \cite{Adin} extended Kalai's work by proving
a combinatorial formula for $\tau_k(\Gamma)$, which we shall not reproduce here,
for every $1\leq k<r$.  Note that when $r=2$, the complex $\Gamma$ is a complete bipartite
graph, and if $|V_i|=1$ for all $i$, then $\Gamma$ is a simplex.  Thus both
\eqref{Kmn} and \eqref{Kalai-Cayley} can be recovered from Adin's formula.

Kalai's and Adin's
beautiful formulas inspired us to look for more results about simplicial spanning
tree enumeration, and in particular to formulate a simplicial version of the Matrix-Tree
Theorem that could be applied to as broad a class of complexes as possible.
Our first main result generalizes the Matrix-Tree Theorem to all APC simplicial complexes.

\setcounter{SMTT-counter}{\value{theorem}}
\begin{theorem}[{\bf Simplicial Matrix-Tree Theorem}]
\label{thm:SMTT}
Let $\Delta$ be a $d$-dimensional APC simplicial complex.  Then:
\begin{enumerate}
\item\label{SMTT-eigenvalues} We have
  $$\pi_d(\Delta) = \frac{\tau_d(\Delta) \tau_{d-1}(\Delta)}{|\HH_{d-2}(\Delta)|^2}.$$
\item\label{SMTT-reducedLapl} Let $U$ be the set of facets of a $(d-1)$-SST of $\Delta$, and let
$L_U$ be the reduced Laplacian obtained by deleting the rows and columns of $L$
corresponding to $U$.  Then
  $$\tau_d(\Delta) = \frac{|\HH_{d-2}(\Delta)|^2}{|\HH_{d-2}(\Delta_U)|^2} \det L_U.$$
\end{enumerate}
\end{theorem}

We will prove these formulas in Section~\ref{SMTT-section}.

In the special case $d=1$, the number $\tau_1(\Delta)$ is just the number of spanning
trees of the graph~$\Delta$, recovering the classical Matrix-Tree Theorem.  When
$d\geq 2$, there can exist spanning trees with finite but nontrivial homology groups
(the simplest example is the real projective plane).  In this case, $\tau_k(\Delta)$ is
greater than the number of spanning trees, because these ``torsion trees'' contribute
more than 1 to the count.  This phenomenon was first observed by
Bolker \cite{Bolker}, and arises also in the study of cyclotomic matroids
\cite{MR1} and cyclotomic polytopes \cite{BH}.

The Weighted Matrix-Tree Theorem also has a simplicial analogue.
Introduce an indeterminate
$x_F$ for each facet (maximal face) $F\in\Delta$, and for every set $T$ of facets
define monomials $x_T = \prod_{F\in T} x_F$ and $X_T=x_T^2$.
Construct the \emph{weighted boundary matrix} $\wbd$ by
multiplying each column of $\bd$ by $x_F$, where $F$ is the facet of $\Delta$
corresponding to that column.  Let $\hat\pi_k$ be the product of the nonzero
eigenvalues of $\Lwud_{\Delta,k-1}$, and let
  $$\hat \tau_k = \hat \tau_k(\Delta) = \sum_{\Upsilon\in\SST_k(\Delta)} |\HH_{k-1}(\Upsilon)|^2 X_\Upsilon.$$

\setcounter{WSMTT-counter}{\value{theorem}}
\begin{theorem}[{\bf Weighted Simplicial Matrix-Tree Theorem}]
\label{thm:WSMTT}
Let $\Delta$ be a $d$-dimensional APC simplicial complex.  Then:
\begin{enumerate}
\item We have\footnote{%
    Despite appearances, there are no missing hats on the right-hand side of this formula!
    Only $\tau_d(\Delta)$ has been replaced with its weighted analogue; $\tau_{d-1}(\Delta)$
    is still just an integer.}
  $$\hat\pi_d(\Delta) = \frac{\hat \tau_d(\Delta) \tau_{d-1}(\Delta)}{|\HH_{d-2}(\Delta)|^2}.$$
\item Let $U$ be the set of facets of a $(d-1)$-SST of $\Delta$, and let $\hat L_U$ be the reduced Laplacian obtained by
deleting the rows and columns of $\hat L$ corresponding to $U$.  Then
  $$\hat \tau_d(\Delta) = \frac{|\HH_{d-2}(\Delta)|^2}{|\HH_{d-2}(\Delta_U)|^2} \det \hat L_U.$$
\end{enumerate}
\end{theorem}

We will prove these formulas in Section~\ref{WSMTT-section}.

Setting $x_F=1$ for all $F$ in Theorem~\ref{thm:WSMTT} recovers Theorem~\ref{thm:SMTT}.  In fact, more is
true; setting $x_F=1$ in the multiset of eigenvalues of the weighted Laplacian (reduced or
unreduced) yields the eigenvalues of the corresponding unweighted Laplacian.  If the complex
$\Delta$ is \emph{Laplacian integral}, that is, its Laplacian matrix has integer eigenvalues,
then we can hope to find a combinatorial interpretation of the factorization of
$\hat \tau_d(\Delta)$ furnished by Theorem~\ref{thm:WSMTT}.  An important class of
Laplacian integral simplicial complexes are the \emph{shifted complexes}.

\subsection{Results on shifted complexes}

Let $p\leq q$ be integers, and let $[p,q]=\{i\in\Zz\st p\leq i\leq
q\}$.  A simplicial complex $\Sigma$ on vertex set $[p,q]$ is
\emph{shifted} if the following condition holds: whenever $i<j$ are
vertices and $F\in\Sigma$ is a face such that $i\not\in F$ and $j\in
F$, then $F\sm\{j\}\cup\{i\}\in\Sigma$.  Equivalently, define the
\emph{componentwise partial order} $\cpleq$ on finite sets of positive
integers as follows: $A\cpleq B$ whenever $A=\{a_1<\dots<a_m\}$,
$B=\{b_1<\dots<b_m\}$, and $a_i\leq b_i$ for all $i$.  Then a complex
is shifted precisely when it is an order ideal with respect to the
componentwise partial order.  (See \cite[chapter~3]{EC1} for general
background on partially ordered sets.)  

Shifted complexes were used by Bj\"orner and Kalai~\cite{BK} to characterize 
the $f$-vectors and Betti numbers of all simplicial complexes.
Shifted complexes are also one of a small handful of classes of simplicial complexes
whose Laplacian eigenvalues are known to be integral.  In particular, Duval and Reiner \cite[Thm.~1.1]{DR} 
proved that the Laplacian eigenvalues of a shifted complex
$\Sigma$ on $[p,q]$ are given by the conjugate of the partition $(d_p,d_{p+1},\dots,d_q)$, where
$d_i$ is the degree of vertex~$i$, that is, the number of facets containing it.  

In the second part of the article,
Sections~\ref{shifted-section}--\ref{corollary-section}, we study 
factorizations of the weighted spanning tree enumerator
of $\Sigma$ under the \emph{combinatorial fine weighting}
  $$x_F = \prod_{i=1}^{k+1} x_{i,v_i}$$
(described in more detail in Section~\ref{shifted-section}),
where $F=\{v_1<\cdots<v_{k+1}\}$ is a $k$-dimensional face of $\Sigma$.  
Thus the term of
$\tau_k(\Sigma)$ corresponding to a particular simplicial spanning tree of $\Sigma$ contains more precise
information than its vertex degrees alone (which can be recovered by further setting
$x_{i,j}=x_j$ for all $i,j$).

For integer sets $A$ and $B$ as above, we call the
ordered pair $(A,B)$ a \emph{critical pair} of
$\Sigma$ if $A\in\Sigma$, $B\not\in\Sigma$, and $B$ covers $A$ in the componentwise order.  That is, $B=\{a_1,\dots,a_{i-1},a_i+1,a_{i+1},\dots,a_m\}$ for some $i\in[m]$.
The \emph{long signature} of $(A,B)$ is the ordered pair $\lsg(A,B)=(S,T)$, where
$S=\{a_1,\dots,a_{i-1}\}$ and $T=[p,a_i]$.
The corresponding \emph{$z$-polynomial} is defined as
  $$
  z(S,T) = \frac{1}{\Up X_S} \sum_{j\in T} X_{S\cup j}
  $$
where $X_S = x_S^2$ for each $S$, and 
the operator $\Up$\; is defined by $\Up(x_{i,j})=x_{i+1,j}$.
(See Section~\ref{z-poly-subsection} for more details, and Example~\ref{ex:bipyramid} for an example.)
The set of critical pairs is especially significant for a shifted
family (and by extension, for a shifted complex).  
Since a shifted family is just an order ideal with respect to the componentwise partial order~$\cpleq$, the critical pairs identify the frontier
between members and non-members of~$\family$ in the
Hasse diagram of~$\cpleq$.  (See Example~\ref{ex:bipyramid} or \cite{Klivans}
for more details.)
 
Thanks to Theorem~\ref{thm:WSMTT}, the enumeration of SST's of a shifted
complex reduces to computing the determinant of the reduced combinatorial 
finely-weighted Laplacian.  We show in Section~\ref{shifted-section} how this computation
reduces to the computation of 
the eigenvalues of the \emph{algebraic finely weighted
Laplacian}.  This modification of the combinatorial fine weighting, designed to
endow the chain groups of $\Sigma$ with the structure of an algebraic chain complex,
is described in detail in Section~\ref{alg-fine-weighting-section}.  Its
eigenvalues turn out to be precisely the $z$-polynomials associated with critical pairs.

\setcounter{zpoly-counter}{\value{theorem}}
\begin{theorem}
\label{zpoly-theorem}
Let $\Sigma$ be a $d$-dimensional shifted complex,
and let $0\leq i\leq d$.  Then the eigenvalues of the algebraic finely
weighted up-down Laplacian $\LLud_{\Delta,i}$ 
are precisely
$\Up^{d-i}(z(S,T))$, where
$(S,T)$ ranges over all long signatures of critical pairs of $i$-dimensional
faces of $\Sigma$.
\end{theorem}

In turn, the $z$-polynomials are the factors of the weighted simplicial spanning tree enumerator $\hat \tau_d$.

\setcounter{shifted-counter}{\value{theorem}}
\begin{theorem}
\label{shifted-theorem}
Let $\Sigma$ be a $d$-dimensional shifted complex with initial vertex~$p$.  Then:
\begin{align*}
\hat{\tau}_d(\Sigma) 
    &= \left( \prod_{ F \in \Lambda_{d-1}} X_{\tilde F} \right)
       \left(  \prod_{(S,T) \in \setlsg(\Delta_d)} \frac{z(S,\tilde T)}{X_{1,p}}  \right) \\
    &= \left( \prod_{ F \in \Lambda_{d-1}} X_{\tilde F} \right)
                   \left( \prod_{(S,T) \in \setlsg(\Delta_d)} 
                               \frac{\sum_{j \in \tilde T} X_{S \cup j}}{X_{\tilde S}} \right)
\end{align*}
where $\tilde F = F \cup\{p\}$; $\Delta = \del_p \Sigma = \{F\sm\{p\}\colon F\in\Sigma\}$;
and $\Lambda = \link_p \Sigma = \{F\colon p\not\in F,\;\tilde F\in\Sigma\}$.
\end{theorem}
Theorems~\ref{zpoly-theorem} and~\ref{shifted-theorem} are proved in Sections~\ref{z-poly-section} and~\ref{count-shifted-section}, respectively.

\begin{example} \label{ex:bipyramid}
As an example to which we will return repeatedly, consider the \emph{equatorial
bipyramid}, the two-dimensional shifted complex $B$ with vertices $[5]$ and facets
123, 124, 125, 134, 135, 234, 235.  A geometric realization
of $B$ is shown in the figure on the left below.  The figure on the right illustrates how
the facets of~$B$ can be regarded as an order ideal.  The boldface lines indicate
critical pairs.

\begin{center}
\resizebox{3in}{1.5in}{\includegraphics{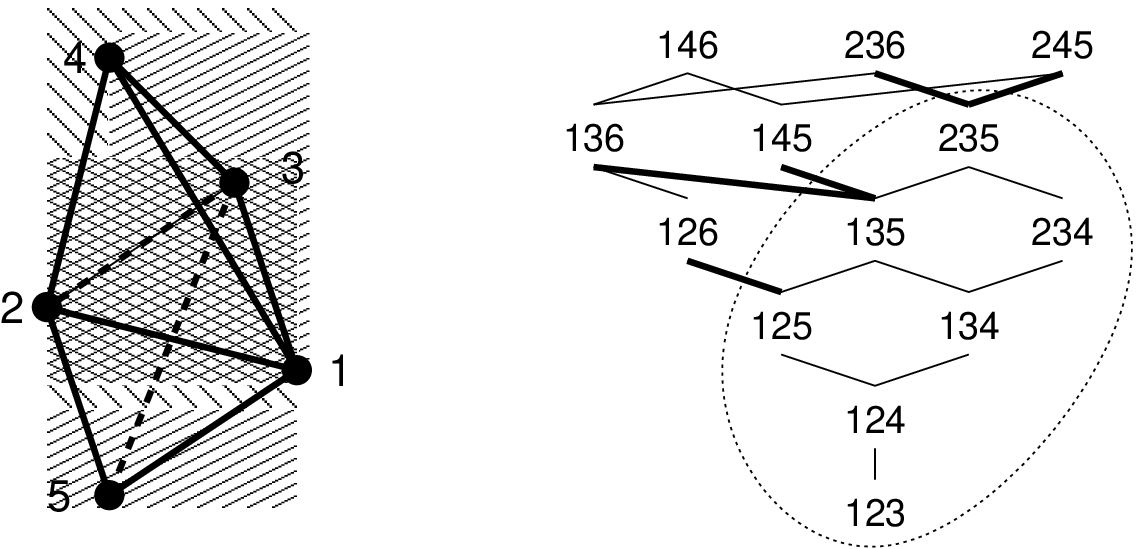}}
\end{center}
The Laplacian eigenvalues corresponding to the critical pairs of $B$ are as follows:
$$
\begin{array}{|c||c|c|c|c|c|} \hline
\text{Critical pair} &  (125,126)  &  (135,136)  &  (135,145) &  (235,236)  &  (235,245)\\ \hline
\text{Eigenvalue}   & z(12,12345) & z(13,12345) & z(1,123)   & z(23,12345) & z(2,123)\\ \hline
\end{array}
$$
To show one of these eigenvalues in more detail, 
$$
z(13, 12345) = \frac{X_{1,1}X_{2,1}X_{3,3} + 
                     X_{1,1}X_{2,2}X_{3,3} +
                     X_{1,1}X_{2,3}X_{3,3} +
                     X_{1,1}X_{2,3}X_{3,4} +
                     X_{1,1}X_{2,3}X_{3,5}}{X_{2,1}X_{3,3}}.
$$
The eigenvalues of this complex are explained in more detail in Section~\ref{sec:example}.  
Its spanning trees are enumerated in Examples~\ref{ex:fine-trees-bipyramid} (fine weighting) 
and~\ref{ex:coarse-trees-bipyramid} (coarse weighting).
\end{example}

We prove Theorem~\ref{zpoly-theorem} by exploiting the recursive structure of shifted
complexes.  As in \cite{DR}, we begin by calculating
the algebraic finely weighted eigenvalues of a near-cone in terms of the eigenvalues
of its link and deletion with respect to its apex (Proposition~\ref{near-cone-recurrence}).
We can then write down a recursive formula (Theorem~\ref{shifted-recurrence}) 
for the nonzero eigenvalues of shifted complexes,
thanks to their characterization as iterated near-cones, simultaneously showing
that these eigenvalues must be of the form $z(S,T)$.  Finally,
we independently establish a recurrence (Corollary~\ref{allsg-shifted}) for the long signatures of
critical pairs of a shifted complex, which coincides with the recurrence for the
$z(S,T)$, thus yielding a bijection between nonzero eigenvalues and critical pairs.  

Corollary~\ref{Swud-shifted} shows what the eigenvalues look like in coarse weighting.  
Passing from weighted to unweighted eigenvalues then easily recovers the Duval-Reiner formula 
for Laplacian eigenvalues of shifted complexes in terms of degree sequences \cite[Thm.~1.1]{DR}.
Similarly, Corollary~\ref{coarse-shifted} gives the enumeration of SST's of a shifted complex in
the coarse weighting.

We are also able
to show that the finely-weighted eigenvalues (though not the 
coarsely-weighted eigenvalues) are enough to recover the 
shifted complex (Corollary~\ref{fine-hearing}), or, in other words, that one can 
``hear the shape'' of a shifted complex.

Several known results can be obtained as consequences of the
general formula of Theorem~\ref{shifted-theorem}.
  \begin{itemize}
 \item The complete $d$-skeleton of a simplex is easily seen to be shifted,
    and applying Theorem~\ref{shifted-theorem} to such complexes recovers Kalai's
    generalization of the Cayley-Pr\"ufer Theorem.
  \item The one-dimensional shifted complexes are precisely the \emph{threshold graphs}, an important
    class of graphs with many equivalent descriptions (see, e.g., \cite{MP}).  When $d=1$,
    Theorem~\ref{shifted-theorem} specializes to the weighted spanning tree enumerator for threshold
    graphs proved by Martin and Reiner \cite[Thm.~4]{MR1} and following from an independent
    result of Remmel and Williamson \cite[Thm.~2.4]{RW}.
  \item Thanks to an idea of Richard Ehrenborg, the formula for threshold graphs can be used to
    recover a theorem of Ehrenborg and van~Willigenburg \cite{EvW}, enumerating spanning trees in
    certain bipartite graphs called \emph{Ferrers graphs} (which are not in general Laplacian integral).
  \end{itemize}
We discuss these corollaries in Section~\ref{corollary-section}.

Some classes of complexes that we think deserve further study include \emph{matroid
complexes}, \emph{matching complexes}, \emph{chessboard complexes} and
\emph{color-shifted} complexes.  The first three kinds of complexes are known
to be Laplacian integral, by theorems of Kook, Reiner and Stanton \cite{KRS},
Dong and Wachs \cite{Dong-Wachs}, and Friedman and Hanlon \cite{Friedman-Hanlon}
respectively.  Every matroid complex is Cohen-Macaulay \cite[\S III.3]{CCA}, hence
APC, while matching complexes and chessboard complexes are APC for certain
values of their defining parameters (see \cite{BLVZ}).
Color-shifted complexes, which are a common generalization of Ferrers
graphs and complete colorful complexes, are not in general Laplacian integral;
nevertheless, their weighted simplicial spanning tree enumerators seem to have
nice factorizations.

It is our pleasure to thank Richard Ehrenborg, Vic Reiner, and Michelle Wachs for many valuable discussions.
We also thank Andrew Crites and an anonymous referee for their careful reading of the manuscript.

\section{Notation and definitions}

\subsection{Simplicial complexes}

Let $V$ be a finite set.  A \emph{simplicial complex on $V$} is a
family $\Delta$ of subsets of $V$ such that
\begin{enumerate}
\item $\0\in\Delta$;
\item If $F\in\Delta$ and $G\subseteq F$, then $G\in\Delta$.
\end{enumerate}
The elements of $V$ are called \emph{vertices} of $\Delta$,
and the faces that are maximal under inclusion are called \emph{facets}.
Thus a simplicial complex is determined by its set of facets.
The \emph{dimension} of a face $F$ is $\dim F=|F|-1$, and the
dimension of $\Delta$ is the maximum dimension of a face (or facet).
The abbreviation $\Delta^d$ indicates that $\dim\Delta=d$.
We say that $\Delta$ is \emph{pure} if all facets have the same dimension;
in this case, a \emph{ridge} is a face of codimension~1, that is,
dimension $\dim\Delta-1$.

We write $\Delta_i$ for the set of $i$-dimensional faces of~$\Delta$,
and set $f_i(\Delta)=|\Delta_i|$.  The \emph{$i$-skeleton}
of $\Delta$ is the subcomplex of all faces of dimension~$\leq i$,
  $$\Delta_{(i)} = \bigcup_{-1\leq j\leq i} \Delta_j,$$
and the \emph{pure $i$-skeleton} of $\Delta$ is the subcomplex
generated by the $i$-dimensional faces, that is,
  $$\Delta_{[i]} = \{F \in \Delta\st  F \subseteq G \text{ for some } G \in \Delta_i\}.$$

We assume that the reader is familiar with
simplicial homology; see, e.g., \cite[\S2.1]{Hatcher}.
Let $\Delta^d$ be a simplicial complex and $-1\leq i\leq d$.
Let $R$ be a ring (if unspecified, assumed
to be $\Zz$), and let $C_i(\Delta)$ be the $i^{th}$
simplicial chain group of $\Delta$, i.e., the free $R$-module with
basis $\{[F] \st F\in\Delta_i\}$.  We denote the
simplicial boundary and coboundary maps respectively by
  \begin{align*}
  \bd_{\Delta,i}  &\;:\; C_i(\Delta) \to C_{i-1}(\Delta),\\
  \cbd_{\Delta,i} &\;:\; C_{i-1}(\Delta) \to C_i(\Delta),
  \end{align*}
where we have identified cochains with chains via the natural
inner product.  We will abbreviate the subscripts in
the notation for boundaries and coboundaries whenever no ambiguity can arise.  
We will often regard $\bd_i$ (resp.\ $\cbd_i$)
as a matrix whose columns and rows (resp.\ rows and columns)
are indexed by $\Delta_i$ and $\Delta_{i-1}$ respectively.
The \emph{$i^{th}$ (reduced) homology group} of $\Delta$ is
$\HH_i(\Delta)=\ker(\bd_i)/\im(\bd_{i+1})$, and the \emph{$i^{th}$
(reduced) Betti number} $\betti_i(\Delta)$ is the rank of the
largest free $R$-module summand of $\HH_i(\Delta)$.

\subsection{Combinatorial Laplacians}

We adopt the notation of \cite{DR} for the Laplacian operators (or, equivalently, matrices)
of a simplicial complex.  We summarize the notation and mention some fundamental
identities here.

We will often work with \emph{multisets} (of eigenvalues or of vertices),
in which each element occurs
with some non-negative integer multiplicity.  For brevity, we drop curly braces and commas
when working with multisets of integers: for instance, 5553 denotes the multiset
in which 5 occurs with multiplicity~three and 3 occurs with multiplicity~one.
The cardinality of a multiset is the sum of the multiplicities of its
elements; thus $|5553|=4$.  We write $\mathbf{a} \zeq \mathbf{b}$ to mean
that the multisets $\mathbf{a}$ and $\mathbf{b}$ differ only in their respective
multiplicities of zero; for instance, $5553\zeq 55530\zeq 555300$.
Of course, $\zeq$ is an equivalence relation.  The union operation $\cup$ on
multisets is understood to add multiplicities: for instance, $5553 \cup 5332 = 55553332$.

For $-1\leq i\leq\dim\Delta$,
define linear operators $\Lud_{\Delta,i}$, $\Ldu_{\Delta,i}$, $\Ltot_{\Delta,i}$ on the
vector space $C_i(\Delta)$ by
\begin{align*}
& \Lud_{\Delta,i}  = \bd_{i+1}\cbd_{i+1}            && (\text{the \emph{up-down Laplacian}}),\\
& \Ldu_{\Delta,i}  = \cbd_i\bd_i                    && (\text{the \emph{down-up Laplacian}}),\\
& \Ltot_{\Delta,i} = \Lud_{\Delta,i} + \Ldu_{\Delta,i} && (\text{the \emph{total Laplacian}}).
\end{align*}

The \emph{spectrum} $\stot_i(\Delta)$ of $\Ltot_{\Delta,i}$ is the
multiset of its eigenvalues (including zero); we define $\sud_i(\Delta)$
and $\sdu_i(\Delta)$ similarly.  Since each Laplacian operator is represented by a symmetric matrix,
it is diagonalizable, so
  $$|\stot_i(\Delta)|=|\sud_i(\Delta)|=|\sdu_i(\Delta)|=f_i(\Delta).$$

The various Laplacian spectra are related by the identities
  \begin{align*}
  \sud_i(\Delta)  &\zeq \sdu_{i+1}(\Delta),\\
  \stot_i(\Delta) &\zeq \sud_i(\Delta) \cup \sdu_i(\Delta)
  \end{align*}
\cite[eqn.~(3.6)]{DR}.  Therefore, each of the three families of multisets
  $$\{\stot_i(\Delta) \st -1\leq i\leq\dim\Delta\},\quad
    \{\sud_i(\Delta) \st -1\leq i\leq\dim\Delta\},\quad
    \{\sdu_i(\Delta) \st -1\leq i\leq\dim\Delta\}$$
determines the other two, and we will feel free to work with whichever
one is most convenient in context.

Combinatorial Laplacians and their spectra have been investigated for
a number of classes of simplicial complexes. In particular, it is
known that chessboard~\cite{Friedman-Hanlon},
matching~\cite{Dong-Wachs}, matroid~\cite{KRS}, and shifted~\cite{DR}
complexes are \emph{Laplacian integral}, i.e., all their Laplacian eigenvalues
are integers.  Understanding which complexes are Laplacian integral is an
open question.  As we will see, Laplacian eigenvalues and spanning
tree enumerators are inextricably linked.

\section{Simplicial spanning trees}

In this section, we generalize the notion of
a spanning tree to arbitrary dimension using simplicial homology,
following Kalai's idea.  Our definition makes sense for any ambient
complex that satisfies the relatively mild APC condition.

\begin{definition} \label{def_sst}
Let $\Delta^d$ be a simplicial complex, and let $k\leq d$.
A \emph{$k$-dimensional simplicial spanning tree} (for short, SST or $k$-SST) of
$\Delta$ is a $k$-dimensional subcomplex $\Upsilon\subseteq\Delta$ such that
$\Upsilon_{(k-1)} = \Delta_{(k-1)}$ and
  \begin{subequations}
  \begin{align}
  \label{acyc-condn}  & \HH_k(\Upsilon) = 0,\\
  \label{conn-condn}  & |\HH_{k-1}(\Upsilon)| < \infty, \quad\text{and}\\
  \label{count-condn} & f_k(\Upsilon) = f_k(\Delta) - \betti_k(\Delta) + \betti_{k-1}(\Delta).
  \end{align}
  \end{subequations}

We write $\SST_k(\Delta)$ for the set of all $k$-SST's of $\Delta^d$,
omitting the subscript if $k=d$.
Note that $\SST_k(\Delta)=\SST_k(\Delta_{(j)})$ for all $j\geq k$.  
\end{definition}

A zero-dimensional SST is just a vertex of $\Delta$.
If $\Delta$ is a 1-dimensional simplicial complex on $n$ vertices---that is, a graph---then
the definition of 1-SST coincides with the usual definition of a spanning tree of a graph:
namely, a subgraph of $\Delta$ which is connected, acyclic, and has $n-1$ edges. 
Next, we give a few examples in higher dimensions.

\begin{example}
\label{ex:sphere}
If $\Delta^d$ is a simplicial sphere (for instance, the boundary
of a simplicial polytope), then deleting any facet of $\Delta$ while keeping its
($d-1$)-skeleton intact produces a $d$-SST.
Therefore $|\SST(\Delta)|=f_d(\Delta)$.
\end{example}

\begin{example}
\label{ex:rp2}
In dimension $>1$, spanning trees need not be $\Zz$-acyclic, merely $\Qq$-acyclic.
For example, let $\Delta$ be a triangulation of the real projective plane,
so that $\dim\Delta=2$, $\HH_1(\Delta,\Zz)\isom\Zz/2\Zz$, and $\HH_1(\Delta,\Qq)=0$.
Then $\Delta$ satisfies the conditions of Definition~\ref{def_sst} and is a
2-SST of itself (in fact, the only such).
\end{example}

\begin{example}
\label{ex:shifted235}
Consider the equatorial bipyramid~$B$ of Example~\ref{ex:bipyramid}.  A 2-SST
of $B$ can be constructed by removing two facets $F,F'$, provided that $F\cap F'$
contains neither of the vertices 4,5.  A simple count shows that there are
15 such pairs $F,F'$, so $|\SST_2(B)|=15$.
\end{example}

Before proceeding any further, we show that Definition~\ref{def_sst} satisfies
a ``two-out-of-three theorem'' akin to that for spanning trees of graphs.

\begin{proposition} \label{two-out-of-three}
Let $\Upsilon\subset\Delta^d$ be a $k$-dimensional subcomplex with $\Upsilon_{(k-1)}=\Delta_{(k-1)}$.
Then any two of the conditions \eqref{acyc-condn}, \eqref{conn-condn},
\eqref{count-condn} together imply the third.
\end{proposition}

\begin{proof}
First, note that
  \begin{equation} \label{skeleton}
  f_\ell(\Upsilon) = f_\ell(\Delta)\quad\text{for } \ell\leq k-1
  \qquad\text{and}\qquad
  \betti_\ell(\Upsilon) = \betti_\ell(\Delta)\quad\text{for } \ell\leq k-2.
  \end{equation}
Next, we use the standard fact that the Euler characteristic $\chi(\Upsilon)$
can be calculated as the alternating sum either of the $f$-numbers or of the Betti numbers.
Thus
  \begin{align}
  \chi(\Upsilon) &= \sum_{i=0}^k (-1)^i f_i(\Upsilon) \notag\\
    &= (-1)^k f_k(\Upsilon) + \sum_{i=0}^{k-1} (-1)^i f_i(\Delta) \notag\\
    &= (-1)^k f_k(\Upsilon) + \chi(\Delta) - (-1)^k f_k(\Delta) \label{Euler:1}
  \end{align}
and on the other hand,
  \begin{align}
  \chi(\Upsilon) &= \sum_{i=0}^k (-1)^i \betti_i(\Upsilon) \notag\\
    &= (-1)^k(\betti_k(\Upsilon)-\betti_{k-1}(\Upsilon)) + \sum_{i=0}^{k-2} (-1)^i \betti_i(\Delta) \notag\\
    &= (-1)^k(\betti_k(\Upsilon)-\betti_{k-1}(\Upsilon)) + \chi(\Delta) - (-1)^k(\betti_k(\Delta)-\betti_{k-1}(\Delta)). 
\label{Euler:2}
  \end{align}
Equating \eqref{Euler:1} and \eqref{Euler:2} gives
  $$f_k(\Upsilon)-f_k(\Delta) = \betti_k(\Upsilon)-\betti_{k-1}(\Upsilon)-\betti_k(\Delta)+\betti_{k-1}(\Delta)$$
or equivalently
  $$\big(f_k(\Upsilon)-f_k(\Delta)+\betti_k(\Delta)-\betti_{k-1}(\Delta)\big)
    -\betti_k(\Upsilon)+\betti_{k-1}(\Upsilon) = 0.$$
Since \eqref{acyc-condn} says that $\betti_k(\Upsilon)=0$ (note that $\HH_k(\Upsilon)$ must be
free abelian) and \eqref{conn-condn} says that $\betti_{k-1}(\Upsilon)=0$, the
conclusion follows.
\end{proof}

\begin{definition}
A simplicial complex $\Delta^d$ is \emph{acyclic in positive codimension}, or APC for short,
if $\betti_j(\Delta) = 0$ for all $j < d$.
\end{definition}

Equivalently, a complex $\Delta^d$ is APC if it has
the homology type of a wedge of zero or more $d$-dimensional spheres.
In particular, any Cohen-Macaulay complex is APC.  The
converse is very far from true, because, for instance, an APC
complex need not even be pure.  For our purposes, the APC
complexes are the ``correct'' simplicial analogues of connected graphs
for the following reason.

\begin{proposition} \label{metaconn}
For any simplicial complex $\Delta^d$, the following are equivalent:
\begin{enumerate}
\item $\Delta$ is APC.
\item $\Delta$ has a $d$-dimensional spanning tree.
\item $\Delta$ has a $k$-dimensional spanning tree for every $k\leq d$.
\end{enumerate}
\end{proposition}

\begin{proof}
It is trivial that (3) implies (2).
To see that (2) implies (1), suppose that $\Delta$ has a $d$-dimensional
spanning tree $\Upsilon$.  Then $\Upsilon_i=\Delta_i$ for all $i\leq d-1$, so
$\HH_i(\Delta)=\HH_i(\Upsilon)=0$ for all $i\leq d-2$.  Moreover, in the diagram
  $$
  \begin{array}{ccccc}
  C_d(\Delta) & \xrightarrow{\bd_{\Delta,d}} & C_{d-1}(\Delta) & \xrightarrow{\bd_{\Delta,d-1}} & C_{d-2}(\Delta) \\
  \bigcup & & \Vert & & \Vert \\
  C_d(\Upsilon) & \xrightarrow{\bd_{\Upsilon,d}} & C_{d-1}(\Upsilon) & \xrightarrow{\bd_{\Upsilon,d-1}} & C_{d-2}(\Upsilon)
  \end{array}
  $$
we have $\ker\bd_{\Delta,d-1} = \ker\bd_{\Upsilon,d-1}$
and $\im\bd_{\Delta,d}\supseteq \im\bd_{\Upsilon,d}$,
so there is a surjection $0=\HH_{d-1}(\Upsilon) \rightarrow \HH_{d-1}(\Delta)$,
implying that $\Delta$ is APC.

To prove (1) implies (3), it suffices to consider the case $k=d$, because
any skeleton of an APC complex is also APC. We can construct a $d$-SST $\Upsilon$ by the following
algorithm.  Let $\Upsilon=\Delta$.  If $\HH_d(\Upsilon)\neq 0$, then there
is some nonzero linear combination of facets of $\Upsilon$ that is mapped to
zero by $\bd_{\Upsilon,d}$.  Let $F$ be one of those facets, and let
$\Upsilon'=\Upsilon\sm\{F\}$.  Then $\betti_d(\Upsilon')=\betti_d(\Upsilon)-1$
and $\betti_i(\Upsilon')=\betti_i(\Upsilon)$ for $i\leq d-2$, and by the Euler
characteristic formula, we have $\betti_{d-1}(\Upsilon')=\betti_{d-1}(\Upsilon)$
as well.  Replacing $\Upsilon$ with $\Upsilon'$ and repeating, we eventually
arrive at the case $\HH_d(\Upsilon)=0$, when $\Upsilon$ is a $d$-SST
of $\Delta$.
\end{proof}

The APC condition is a fairly mild one.  For instance, any
$\Qq$-acyclic complex is clearly APC (and is its own unique
SST), as is any Cohen-Macaulay complex (in particular,
any shifted complex).

\section{Simplicial analogues of the Matrix-Tree Theorem}
\label{SMTT-section}

We now explain how to enumerate simplicial spanning trees of a complex
using its Laplacian.
Throughout this section, let $\Delta^d$ be an APC simplicial complex
on vertex set $[n]$.  For $k\leq d$, define
\begin{align*}
  \pi_k &=\pi_k(\Delta) = \prod_{0\neq\lambda\in\sud_{k-1}(\Delta)} \lambda,&
  \tau_k &=\tau_k(\Delta) = \sum_{\Upsilon\in\SST_k(\Delta)} |\HH_{k-1}(\Upsilon)|^2.
\end{align*}

We are interested in the relationships between these two
families of invariants.  When $d=1$, the relationship is given by Theorem~\ref{CMTT}.
In the notation
just defined, part (1) of that theorem says that $\tau_1=\pi_1/n$, and part (2) says that
$\tau_1=\det L_i$ (i.e., the determinant of the reduced Laplacian obtained from $\Lud_{\Delta,0}$
by deleting the row and column corresponding to any vertex $i$).

The results of this section generalize both parts of the Matrix-Tree Theorem from graphs to
all APC complexes~$\Delta^d$.
Our arguments are closely based on those used by Kalai \cite{Kalai} and Adin
\cite{Adin} to enumerate SST's of skeletons of simplices
and of complete colorful complexes.

We begin by setting up some notation.  Abbreviate
$\betti_i=\betti_i(\Delta)$, $f_i=f_i(\Delta)$, and $\bd=\bd_{\Delta,d}$.
Let $T$ be a set of facets of $\Delta$ of cardinality $f_d-\betti_d+\betti_{d-1}=f_d-\betti_d$,
and let $S$ be a set of ridges such that $|S|=|T|$.  Define
$$
\Delta_T = T \cup \Delta_{(d-1)},\qquad
\bar{S} = \Delta_{(d-1)}\setminus S,\qquad
\Delta_{\bar{S}} = \bar{S} \cup \Delta_{(d-2)},
$$
and let $\bd_{S,T}$ be the square submatrix of $\bd$ with rows indexed by $S$ and columns
indexed by $T$.

\begin{proposition}  \label{nonsingular-criterion}
The matrix $\bd_{S,T}$ is nonsingular if and only if $\Delta_T\in\SST_d(\Delta)$ and $\Delta_{\bar{S}}\in
\SST_{d-1}(\Delta)$.
\end{proposition} 

\begin{proof}
We may regard $\bd_{S,T}$ as the top boundary map of the $d$-dimensional
relative complex $\Gamma =(\Delta_T,\Delta_{\bar{S}})$.  So $\bd_{S,T}$ is nonsingular
if and only if $\HH_d(\Gamma)=0$.  Consider the long exact sequence
  \begin{equation} \label{long-exact}
  0 \rightarrow \HH_d(\Delta_{\bar{S}}) \rightarrow \HH_d(\Delta_T) \rightarrow \HH_d(\Gamma)
    \rightarrow \HH_{d-1}(\Delta_{\bar{S}}) \rightarrow \HH_{d-1}(\Delta_T) \rightarrow \HH_{d-1}(\Gamma)
    \rightarrow \cdots
  \end{equation}
If $\HH_d(\Gamma)\neq 0$, then $\HH_d(\Delta_T)$ and $\HH_{d-1}(\Delta_{\bar{S}})$
cannot both be zero.  This proves the ``only if'' direction.
 
If $\HH_d(\Gamma)=0$, then $\HH_d(\Delta_{\bar{S}})=0$ (since $\dim\Delta_{\bar{S}} = d-1$), so
\eqref{long-exact} implies $\HH_d(\Delta_T)=0$.  Therefore $\Delta_T$ is a $d$-tree,
because it has the correct number of facets.  Hence $\HH_{d-1}(\Delta_T)$ is finite.
Then \eqref{long-exact} implies that $\HH_{d-1}(\Delta_{\bar{S}})$ is finite.  In fact, it is zero
because the top homology group of any complex must be
torsion-free.  Meanwhile, $\Delta_{\bar{S}}$ has the correct number of facets to be a
$(d-1)$-SST of $\Delta$, proving the ``if'' direction.
\end{proof}

\begin{proposition} \label{detD-formula}
If $\bd_{S,T}$ is nonsingular, then
  $$|\det \bd_{S,T}| = 
    \frac{|\HH_{d-1}(\Delta_T)| \cdot |\HH_{d-2}(\Delta_{\bar{S}})|}{ |\HH_{d-2}(\Delta_T)|} = 
    \frac{|\HH_{d-1}(\Delta_T)| \cdot |\HH_{d-2}(\Delta_{\bar{S}})|}{ |\HH_{d-2}(\Delta)|}.$$
\end{proposition}

\begin{proof}
As before, we interpret $\bd_{S,T}$ as the boundary map of the relative complex
$\Gamma=(\Delta_T,\Delta_{\bar{S}})$.  So $\bd_{S,T}$ is a map from $\Z^{|T|}$ to
$\Z^{|T|}$, and $\Z^{|T|} / \bd_{S,T}(\Z^{|T|})$ is a finite abelian group
of order $|\det\bd_{S,T}|$.  On the other hand, since $\Gamma$ has no faces
of dimension $\leq d-2$, its lower boundary maps are
all zero, so $|\det\bd_{S,T}| = |\HH_{d-1}(\Gamma)|$.  Since  $\HH_{d-2}(\Delta_T)$ is finite,
the desired result now follows from the piece
  \begin{equation} \label{long-exact:2}
  0 \rightarrow \HH_{d-1}(\Delta_T) \rightarrow \HH_{d-1}(\Gamma) \rightarrow \HH_{d-2}(\Delta_{\bar{S}}) 
\rightarrow \HH_{d-2}(\Delta_T) \rightarrow 0
  \end{equation}
of the long exact sequence \eqref{long-exact}.
\end{proof}

We can now prove the first version of the Simplicial Matrix-Tree Theorem,
relating the quantities $\pi_d$ and $\tau_d$.  Abbreviate $L=\Lud_{\Delta,d-1}$.

\setcounter{savesection}{\value{section}}
\setcounter{section}{\value{intro-section-counter}}
\setcounter{savetheorem}{\value{theorem}}
\setcounter{theorem}{\value{SMTT-counter}}
\begin{theorem}[{\bf Simplicial Matrix-Tree Theorem}]
Let $\Delta^d$ be an APC simplicial complex.  Then:
\begin{enumerate}
\item We have
  $$\pi_d(\Delta) = \frac{\tau_d(\Delta) \tau_{d-1}(\Delta)}{|\HH_{d-2}(\Delta)|^2}.$$
\item Let $U$ be the set of facets of a $(d-1)$-SST of $\Delta$, 
and let $L_U$ denote the reduced Laplacian\footnote{%
  A warning: This notation for reduced Laplacians specifies which rows and columns to
  \emph{exclude} (in analogy to the notation $L_i$ in the statement of Theorem~\ref{CMTT}),
  in contrast to the notation $\bd_{S,T}$ for restricted boundary maps, which specifies
  which rows and columns to \emph{include}.}
obtained by deleting the rows and columns of
$L$ corresponding to $U$. 
Then 
  $$\tau_d(\Delta) = \frac{|\HH_{d-2}(\Delta)|^2}{|\HH_{d-2}(\Delta_U)|^2} \det L_U.$$
\end{enumerate}
\end{theorem}
\setcounter{section}{\value{savesection}}
\setcounter{theorem}{\value{savetheorem}}

\begin{proof}[Proof of Theorem~\ref{thm:SMTT}~(1)]
The Laplacian $L$ is a square matrix with $f_{d-1}$ rows and columns,
and rank $f_d-\betti_d=f_d-\betti_d+\betti_{d-1}$ (because $\Delta$
is APC).  Let $\chi(L;y)=\det(yI-L)$ be
its characteristic polynomial (where $I$ is an identity matrix), so that $\pi_d(\Delta)$, the product of
the nonzero eigenvalues of $L$, is given (up to sign) by the coefficient
of $y^{f_{d-1}-f_d+\betti_d}$ in $\chi(L;y)$.  Equivalently,
  \begin{equation} \label{pi-formula}
  \pi_d ~= \sum_{\substack{S\subset\Delta_{d-1}\\ |S|=\rank L}}\!\!\!\!\! \det L_U
        ~= \sum_{\substack{S\subset\Delta_{d-1}\\ |S|=f_d-\betti_d}}\!\!\!\!\!\det L_U
  \end{equation}
where $U=\Delta_{d-1}\sm S$ in each summand.  By the Binet-Cauchy formula, we have
  \begin{equation} \label{LT-formula}
  \det L_U ~= \sum_{\substack{T\subset\Delta_d\\ |T|=|S|}}\!\!(\det\bd_{S,T})(\det\cbd_{S,T})
           ~= \sum_{\substack{T\subset\Delta_d\\ |T|=|S|}}\!\!(\det\bd_{S,T})^2.
  \end{equation}
Combining \eqref{pi-formula} and \eqref{LT-formula}, applying
Proposition~\ref{nonsingular-criterion}, and interchanging the sums, we obtain
  $$\pi_d = \sum_{T:\Delta_T\in\SST_d(\Delta)}\ \ \sum_{S:\Delta_{\bar{S}}\in\SST_{d-1}(\Delta)} (\det\bd_{S,T})^2$$
and now applying Proposition~\ref{detD-formula} yields
  \begin{align*}
    \pi_d &=  \sum_{T:\Delta_T\in\SST_d(\Delta)}\ \ \sum_{S:\Delta_{\bar{S}}\in\SST_{d-1}(\Delta)}
    \left(\frac{|\HH_{d-1}(\Delta_T)|\cdot|\HH_{d-2}(\Delta_{\bar{S}})|}{|\HH_{d-2}(\Delta)|}\right)^2\\\\
    &= \frac{ \left(\displaystyle\sum_{T:\Delta_T\in\SST_d(\Delta)} |\HH_{d-1}(\Delta_T)|^2\right)
              \left(\displaystyle\sum_{S:\Delta_{\bar{S}}\in\SST_{d-1}(\Delta)} |\HH_{d-2}(\Delta_{\bar{S}})|\right)}
            {|\HH_{d-2}(\Delta)|^2}
  \end{align*}
as desired.
\end{proof}

In order to prove the ``reduced Laplacian'' part of Theorem~\ref{thm:SMTT},
we first check that when we delete the rows of $\bd$ corresponding to a
$(d-1)$-SST, the resulting reduced Laplacian has the correct size,
namely, that of a $d$-SST.

\begin{lemma} \label{right-size}
Let $U$ be the set of facets of a $(d-1)$-SST of $\Delta$, and let
$S = \Delta_{d-1} \sm U$.
Then $|S|=f_d(\Delta)-\betti_d(\Delta)$, the number of facets of a $d$-SST
of $\Delta$.
\end{lemma}

\begin{proof}
Let $\Gamma=\Delta_{(d-1)}$.  By Proposition~\ref{two-out-of-three}
and the observation \eqref{skeleton},
$|U|=f_{d-1}(\Gamma)-\betti_{d-1}(\Gamma)+\betti_{d-2}(\Gamma)=
f_{d-1}(\Delta)-\betti_{d-1}(\Gamma)$, so $|S|=\betti_{d-1}(\Gamma)$.
  The Euler characteristics of
$\Delta$ and $\Gamma$ are
  \begin{align*}
  \chi(\Delta) = \sum_{i=0}^d (-1)^i f_i(\Delta)     &= \sum_{i=0}^d (-1)^i \betti_i(\Delta),\\
  \chi(\Gamma) = \sum_{i=0}^{d-1} (-1)^i f_i(\Gamma) &= \sum_{i=0}^{d-1} (-1)^i \betti_i(\Gamma).
  \end{align*}
By \eqref{skeleton}, we see that
  $$
  \chi(\Delta)-\chi(\Gamma)
    ~=~ (-1)^d f_d(\Delta)
    ~=~ (-1)^d \betti_d(\Delta) + (-1)^{d-1}\betti_{d-1}(\Delta) - (-1)^{d-1}\betti_{d-1}(\Gamma)
  $$
from which we obtain $f_d(\Delta) = \betti_d(\Delta) - \betti_{d-1}(\Delta) + \betti_{d-1}(\Gamma)$.
Since $\Delta$ is APC, we have $\betti_{d-1}(\Delta)=0$, so
$|S| = \betti_{d-1}(\Gamma) = f_d(\Delta) - \betti_d(\Delta)$ as desired.
\end{proof}

\begin{proof}[Proof of Theorem~\ref{thm:SMTT}~(2)]
By the Binet-Cauchy formula, we have
$$
\det L_U = \sum_{T:\ |T|=|S|} (\det\bd_{S,T})(\det\cbd_{S,T})
= \sum_{T:\ |T|=|S|} (\det\bd_{S,T})^2.
$$

By Lemma~\ref{right-size} and Proposition~\ref{nonsingular-criterion}, $\bd_{S,T}$ is nonsingular exactly
when $\Delta_T\in\SST_d(\Delta)$.  Hence Proposition~\ref{detD-formula} gives
  \begin{align*}
  \det L_U &= \sum_{T:\Delta_T \in \SST_d(\Delta)}
             \left(\frac{|\HH_{d-1}(\Delta_T)|\cdot|\HH_{d-2}(\Delta_U)|}
             {|\HH_{d-2}(\Delta)|}\right)^2 \\
           &= \frac{|\HH_{d-2}(\Delta_U)|^2}{|\HH_{d-2}(\Delta)|^2}
             \sum_{T:\Delta_T\in\SST_d(\Delta)} |\HH_{d-1}(\Delta_T)|^2
           ~=~ \frac{|\HH_{d-2}(\Delta_U)|^2}{|\HH_{d-2}(\Delta)|^2} \tau_d(\Delta),
  \end{align*}
which is equivalent to the desired formula.
\end{proof}

\begin{remark}
Suppose that $\HH_{d-2}(\Delta)=0$ (for example, if
$\Delta$ is Cohen-Macaulay).
Then the two versions of Theorem~\ref{thm:SMTT} assert that
  $$\tau_d = \frac{\pi_d}{\tau_{d-1}} = \frac{\det L_U}{|\HH_{d-2}(\Delta_U)|^2},$$
from which it is easy to recognize the
two different versions of the classical Matrix-Tree Theorem, Theorem~\ref{CMTT}.
(A graph is Cohen-Macaulay as a simplicial complex if and only
if it is connected.)
Moreover, the recurrence $\tau_d = \pi_d/\tau_{d-1}$ leads to
an expression for $\tau_d$ as an alternating product of eigenvalues:
  \begin{equation} \label{reidemeister}
  \tau_d = \frac{\pi_d \pi_{d-2} \cdots }{\pi_{d-1} \pi_{d-3} \cdots} = \prod_{k=0}^d \pi_k^{(-1)^{d-k}}.
  \end{equation}
This formula is reminiscent of the Reidemeister torsion of a chain complex or CW-complex
(although $\tau_d$ is of course not a topological invariant); see, e.g., \cite{Turaev}.
Furthermore, \eqref{reidemeister} is in practice an efficient way to calculate $\tau_d$.
\end{remark}

\begin{example}
\label{ex:altproduct}
For the equatorial bipyramid $B$, we have
  $$\pi_0(B)=5,\qquad \pi_1(B)=5\cdot 5\cdot 5\cdot 3=375,\qquad \pi_2(B) = 5\cdot 5\cdot 5\cdot 3\cdot 3 = 1125.$$
These numbers can be checked by computation, and also follow from the Duval-Reiner formula
for Laplacian eigenvalues of a shifted complex.  Applying the
alternating product formula \eqref{reidemeister} yields
  $$\tau_0(B)=5,\qquad \tau_1(B) = 375/5 = 75,\qquad \tau_2(B) = \frac{1125\cdot 5}{375} = 15.$$
Indeed, $\tau_0(B)$ is the number of vertices.  
Cayley's formula implies that deleting any one edge~$e$ from $K_n$ yields a graph with $(n-2)n^{n-3}$
spanning trees (because $e$ itself belongs to $(n-1)/\binom{n}{2}$ of the spanning trees of $K_n$), and the 1-skeleton $B_{(1)}$
is such a graph with $n=5$, so $\tau_1(B)=75$.  Finally, we have seen in Example~\ref{ex:shifted235}
that $\tau_2(B)=15$.  
\end{example}

\section{Weighted enumeration of simplicial spanning trees}
\label{WSMTT-section}

We can obtain much finer enumerative information 
by labeling the facets of a complex with indeterminates, so that
the invariant $\tau_k$ becomes a generating function for
its SST's.

Let $\Delta^d$ be an APC simplicial complex,
and let $\bd=\bd_{\Delta,d}$.  Introduce an indeterminate $x_F$
for each facet $F$ of maximum dimension, and let $X_F=x_F^2$.
For every $T\subseteq\Delta_d$, let $x_T = \prod_{F\in T} x_F$ and let $X_T=x_T^2$.
To construct the \emph{weighted boundary matrix} $\wbd$ from $\bd$,
multiply each column of $\wbd$ by $x_F$, where $F$ is the facet of $\Delta$
corresponding to that column.  The \emph{weighted coboundary} $\wcbd$ is the
transpose of $\wbd$.  
We can now define weighted versions of Laplacians, the various submatrices of
the boundary and coboundary matrices used in Section~\ref{SMTT-section}, and the
invariants $\pi_k$ and $\tau_k$.  We will notate each weighted invariant by placing
a hat over the symbol for the corresponding unweighted quantity.  Thus
$\hat\pi_k$ is the product of the nonzero eigenvalues of $\Lwud_{\Delta,k-1}$,
and
  $$\hat \tau_k = \hat \tau_k(\Delta) = \sum_{\Upsilon\in\SST_k(\Delta)} |\HH_{k-1}(\Upsilon)|^2 X_\Upsilon.$$
To recover any unweighted quantity from its weighted analogue, set $x_F=1$ for all $F\in\Delta_d$.

\begin{proposition} \label{weighted-tools}
Let $T\subset\Delta_d$ and $S\subset\Delta_{d-1}$, with $|T|=|S|=f_d-\betti_d$.
Then $\det\hat\bd_{S,T} = x_T \det \bd_{S,T}$ is nonzero
if and only if $\Delta_T\in\SST_d(\Delta)$ and $\Delta_{\bar S}\in\SST_{d-1}(\Delta)$.  In that case,
  \begin{equation} \label{weighted-detD}
  \pm\det \hat \bd_{S,T} =
  \frac{|\HH_{d-1}(\Delta_T)| \cdot |\HH_{d-2}(\Delta_{\bar S})|}{ |\HH_{d-2}(\Delta_T)|} x_T  = 
  \frac{|\HH_{d-1}(\Delta_T)| \cdot |\HH_{d-2}(\Delta_{\bar S})|}{ |\HH_{d-2}(\Delta)|} x_T.
  \end{equation}
\end{proposition}

\begin{proof}
The first claim follows from Proposition~\ref{nonsingular-criterion}, and the second
from Proposition~\ref{detD-formula}.
\end{proof}

It is now straightforward to adapt the proofs of both parts of
Theorem~\ref{thm:SMTT} to the weighted setting.  For convenience,
we restate the result.  Let $\hat L=\Lwud_{\Delta,d-1}$.

\setcounter{savesection}{\value{section}}
\setcounter{section}{\value{intro-section-counter}}
\setcounter{savetheorem}{\value{theorem}}
\setcounter{theorem}{\value{WSMTT-counter}}
\begin{theorem}[{\bf Weighted Simplicial Matrix-Tree Theorem}]
Let $\Delta^d$ be an APC simplicial complex.  Then:
\begin{enumerate}
\item We have
  $$\hat\pi_d(\Delta) = \frac{\hat \tau_d(\Delta) \tau_{d-1}(\Delta)}{|\HH_{d-2}(\Delta)|^2}.$$
\item Let $U$ be the set of facets of a $(d-1)$-SST of $\Delta$, and let $\hat L_U$ be the reduced Laplacian obtained by
deleting the rows and columns of $\hat L$ corresponding to $U$.  Then
  $$\hat \tau_d(\Delta) = \frac{|\HH_{d-2}(\Delta)|^2}{|\HH_{d-2}(\Delta_U)|^2} \det \hat L_U.$$
\end{enumerate}
\end{theorem}
\setcounter{section}{\value{savesection}}
\setcounter{theorem}{\value{savetheorem}}

\begin{proof}
For assertion~(1), we use a weighted version of the argument of part~(1) of Theorem~\ref{thm:SMTT}.
By the Binet-Cauchy formula and Proposition~\ref{weighted-tools}, we have
  \begin{align*}
  \hat\pi_d
    &=  \sum_{S\subset\Delta_{d-1}} \sum_{\substack{T\subset\Delta_d\\ |T|=|S|}}\!\!(\det\wcbd_{T,S})(\det\wbd_{S,T})
    = \sum_S \sum_T (\det\wbd_{S,T})^2\\\\
    &= \sum_{T:\Delta_T\in\SST_d(\Delta)} \sum_{S:\Delta_{\bar{S}}\in\SST_{d-1}(\Delta)} (\det \wbd_{S,T})^2\\
    &= \sum_{T:\Delta_T\in\SST_d(\Delta)} \sum_{S:\Delta_{\bar{S}}\in\SST_{d-1}(\Delta)}
       \left(\frac{|\HH_{d-1}(\Delta_T)|\cdot|\HH_{d-2}(\Delta_{\bar{S}})|}{|\HH_{d-2}(\Delta)|}\right)^2 X_T
    = \frac{\hat \tau_d(\Delta) \tau_{d-1}(\Delta)}{|\HH_{d-2}(\Delta)|^2}.
  \end{align*}

The proof of assertion~(2) of the theorem is identical to that of part~(2) of Theorem~\ref{thm:SMTT}, using
Proposition~\ref{weighted-tools} instead of Proposition~\ref{detD-formula}.
\end{proof}

\begin{example}
\label{ex:weightedsst235}
We return to the equatorial bipyramid $B$ of Example~\ref{ex:bipyramid}.
Weight each facet $F = \{i,j,k\}$ by the monomial $x_F = x_ix_jx_k$.
Let $U = \{12, 13, 14, 15\}$ be the facets of a 1-SST of $B_{(1)}$.
Then the reduced Laplacian $\hat{L}_U$ is 
{\small
\begin{displaymath}
\left(
\begin{array}{ccccc}
x_2x_3(x_1+x_4+x_5) & -x_2 x_3 x_5    & x_2 x_3 x_4     & x_2 x_3 x_5     & -x_2 x_3 x_4\\
-x_2 x_3 x_4        & x_2x_5(x_1+x_3) & 0               & -x_2 x_3 x_5    & 0\\
 x_2 x_3 x_4        & 0               & x_3x_4(x_1+x_2) & 0               & -x_2 x_3 x_4\\
 x_2 x_3 x_5        & -x_2 x_3 x_5    & 0               & x_3x_5(x_1+x_2) & 0\\
-x_2 x_3 x_4        & 0               & -x_2 x_3 x_4    & 0               &  x_2x_4(x_1+x_3)
\end{array}
\right)
\end{displaymath}
}
and the generating function for 2-SST's by their degree sequences is
  $$\hat\tau_2(B) = \det \hat{L}_S =
    \sum_{\Upsilon\in\SST(B)} \prod_{i\in[5]} x_i^{\deg_B(i)} =
    x_1^3 x_2^3x_3^3x_4^2x_5^2(x_1 + x_2 + x_3)(x_1 + x_2 + x_3 + x_4 + x_5)$$
where $\deg_B(i)$ means the number of facets of~$B$ containing vertex~$i$.  Setting
$x_i=1$ for every $i$ recovers the unweighted equality $\tau_2(B)=15$ (see
Examples~\ref{ex:shifted235} and~\ref{ex:altproduct}).
\end{example}

\section{Shifted complexes}
\label{shifted-section}
\subsection{General definitions}

In the next several sections of the paper, we apply the tools just developed to the
important class of
\emph{shifted complexes}.  We begin by reviewing some standard facts about shifted complexes
and shifted families; for more details, see, e.g.,~\cite{Kalai_Shift}.

Let $k$ be an integer.  A \emph{$k$-set} is a set of integers of cardinality~$k$. A \emph{$k$-family}
is a set of $k$-sets (for example, the set of $(k-1)$-dimensional faces of a simplicial complex).
The \emph{componentwise partial order} $\cpleq$ on $k$-sets of integers is defined as follows:
if $A=\{a_1<a_2<\cdots<a_{k}\}$ and $B=\{b_1<b_2<\cdots<b_{k}\}$, then 
$A \cpleq B$ if $a_j \leq b_j$ for all $j$.  A $k$-family $\family$ is 
\emph{shifted} if $B \in \family$ and $A \cpleq B$ together imply $A \in \family$.
Equivalently, $\family$ is shifted if it is an order ideal with respect to the
componentwise partial order.  A simplicial complex $\Sigma$ is \emph{shifted} if 
$\Sigma_i$ is shifted for all $i$.  Accordingly, we may specify a shifted complex by the
list of its facets that are maximal with respect to $\cpleq$,
writing $\Sigma=\shgen{F_1,\dots,F_n}$.  For example, the bipyramid of Example~\ref{ex:bipyramid}
is the shifted complex $\shgen{235}$.  We will not lose any generality by assuming
that the vertex set for every shifted complex we encounter is an integer interval
$[p,q] = \{p,p+1,\ldots,q\}$; in particular, we will use the symbol~$p$ throughout
for the vertex with the smallest index.

The \emph{deletion} and \emph{link} of $\Sigma$ with respect to~$p$ are
defined to be the subcomplexes
  \begin{align*}
  \Delta &= \del_p \Sigma = \{F\sm\{p\}\colon F\in\Sigma\},\\
  \Lambda &= \link_p \Sigma = \{F\colon p\not\in F,\;F\cup\{p\}\in\Sigma\}.
  \end{align*}
It is easy to see that the deletion and link of a shifted complex on $[p,q]$
are themselves shifted complexes\footnote{%
    This is also true for the deletion and link with respect to any vertex, not just $p$, but then the resulting vertex
    set is no longer a set of consecutive integers.  Since we will not have any need to take the deletion and link with
    respect to any vertex other than $p$, we won't worry about that, and instead enjoy the resulting simplicity of
    specifying the new minimal vertex of the deletion and link.}
on $[p+1,q]$.

A complex $\Sigma$ on vertex set $V$ is called a \emph{near-cone with apex~$p$} if
it has the following property: if $F\in\del_p\Sigma$ and $v\in F$, then
$F\sm \{v\}\in\link_{p}\Sigma$ (equivalently, $F\sm \{v\} \cup \{p\} \in\Sigma$).
It is easy to see that a shifted complex on $[p,q]$ is a near-cone
with apex $p$.  Bj\"orner and Kalai~\cite[Theorem 4.3]{BK} showed that the Betti numbers
of a shifted complex $\Sigma$ (indeed, of a near-cone) with initial vertex $p$ are given by
\begin{equation}\label{beta-shifted}
\betti_i(\Sigma) = \abs{\{F \in \Sigma_i\st p \not\in F,\ F \dju \{p\} \not\in \Sigma\}}.
\end{equation}

\subsection{The combinatorial fine weighting}

Let $\{x_{i,j}\}$ be a set of indeterminates, indexed by integers $i,j$. Let $\fld$ be the field of
rational functions in the $x_{i,j}$ with coefficients in $\Cc$ (or in any other field
of characteristic zero).
Since these indeterminates will often appear squared, we set $X_{i,j}=x_{i,j}^2$.
The \emph{combinatorial fine weighting} assigns to 
a multiset of vertices $S=\{i_1\leq i_2\leq\dots\leq i_m\}$
the monomials
  \begin{equation} \label{x-monomial}
  x_S = x_{1,i_1} x_{2,i_2} \cdots x_{m,i_m} \quad\text{and}\quad
  X_S = X_{1,i_1} X_{2,i_2} \cdots X_{m,i_m}.
  \end{equation}
Our goal is to describe the generating function
  $$\hat \tau_d(\Sigma) = \sum_{\Upsilon\in\SST(\Sigma)} |\HH_{d-1}(\Sigma,\Zz)|^2 X_\Upsilon$$
of a shifted complex $\Sigma$, where, for each simplicial spanning tree $\Upsilon$, the monomial
  $$X_\Upsilon = \prod_{\text{facets } F\in \Upsilon} X_F$$
records both the number of facets of~$\Upsilon$ containing each vertex of~$\Sigma$,
as well as the order in which the vertices appear in facets.

Define the ``raising operator'' $\Up$\;
by $\Up x_{i,j}=x_{i+1,j}$ for $i\leq d$ and $\Up x_{d+1,j}=0$.
We extend $\Up\;$ linearly and multiplicatively to an operator on all of $\fld$.
The raising operator can also be applied to a $\fld$-linear operator $f$ by
the rule
  \begin{equation} \label{Up-operator}
  (\Up f)(V) = \Up(f(\Up^{-1}(V)))
  \end{equation}
for any vector $V$ over $\fld$.
The $a^{th}$ iterate of $\Up$\; is denoted $\Up^a$.

The following identities will be useful.
Let $\tilde S = S\cup \{p\}$, where $\cup$ denotes the union as multisets,
so that the multiplicity of $p$ in $\tilde S$ is one more than its multiplicity
in $S$.  Then, for all integers $a,j$,
  \begin{subequations}
  \begin{equation} \label{X-ident}
  \Up^a x_{1,p}\,\cdot \Up^{a+1}x_{S\cup j}
  ~=~ \Up^a\big(x_{1,p}\,\cdot\Up x_{S\cup j}\big)
  ~=~ \Up^a x_{\tilde S\cup j}
  \end{equation}
and
  \begin{equation} \label{tilde:monom}
  \frac{\Up^a x_{\tilde S}}{\Up^{a+1} x_S}
  ~=~ x_{a+1,p}
  ~=~ \Up^a x_{1,p}.
  \end{equation}
  \end{subequations}
The same identities hold if $x$ is replaced with $X$.

Now, define the \emph{combinatorially finely weighted simplicial boundary map} 
of $\Sigma$
as the homomorphism $\wbd=\wbd_{\Sigma,i}:C_i(\Sigma)\to C_{i-1}(\Sigma)$ which
acts on generators $[F]$ (for $F\in\Sigma_i$) by
  \begin{equation} \label{fine-boundary}
  \wbd[F] = \sum_{j\in F} \sign(j,F) \Up^{d-i}x_F\; [F\sm j].
  \end{equation}
Here we have set $\sign(v,F)=(-1)^{j+1}$
if $v$ is the $j^{th}$ smallest vertex of $F$, and $\sign(v, F)=0$ if $v\not\in F$.
Similarly define the \emph{finely weighted simplicial coboundary map}
$\wcbd=\wcbd_{ F,i+1}:C_i(\Sigma)\to C_{i+1}(\Sigma)$ by
  \begin{equation} \label{fine-coboundary}
  \wcbd[F] = \sum_{j\in V\sm F} \sign(j, F\cup j) \Up^{d-i-1}x_{F\cup j}\; [F\cup j].
  \end{equation}

These maps do \emph{not} make the chain groups of $\Sigma$ into an algebraic chain 
complex, because $\wbd\wbd$ and $\wcbd\wcbd$ do not vanish in general.  (We will
fix this problem in Section~\ref{alg-fine-weighting-section}.)
On the other hand, they
have combinatorial significance, because we will be able to apply part~(2) of 
Theorem~\ref{thm:WSMTT} to the \emph{finely weighted up-down Laplacian} 
$\Lwud=\wbd_d\wcbd_d$.  This Laplacian may be regarded as a matrix
whose rows and columns are indexed by $\Sigma_{d-1}$.
It is not hard to check that for each $F,G\in\Sigma_{d-1}$,
the corresponding entry of $\Lwud$ is 
\begin{equation}\label{Lwud-entries}
(\Lwud)_{FG} = 
  \begin{cases}
    \sign(j,H)\; \sign(i,H)\; X_H
       &\text{ if $H = F\cup j = G\cup i\in\Sigma$,}\\
    \sum\limits_{j\colon F\cup j\in\Sigma} X_{F\cup j}
       &\text{ if $F=G$,} \\
    \quad0 &\text{ otherwise.}
  \end{cases}
\end{equation}

Let $U$ be the simplicial spanning tree of $\Sigma_{(d-1)}$
consisting of all ridges containing vertex~$p$.  (This subcomplex is an
SST because it has a complete $(d-2)$-skeleton and is a cone over $p$,
hence contractible.)
Let $\Lwud_U$ be the reduced Laplacian obtained
from $\Lwud$ by deleting the corresponding rows and columns, so that the remaining rows
and columns are indexed by the facets of $\Lambda = \link_p\Sigma$.  Then
part~(2) of Theorem~\ref{thm:WSMTT} asserts that
$\hat \tau_d(\Sigma) = \det \Lwud_U$.

Let $N=N(\Sigma)$ be the matrix obtained from $\Lwud_U$ by dividing each row $F$ by
$\Up x_F$ and dividing each column $G$ by $\Up x_G$.
The $(F,G)$ entry of $N$ is thus
\begin{equation}\label{N-entries}
N_{FG} = 
  \begin{cases}
    \sign(j,H)\; \sign(i,H)\; \dfrac{X_H}{\Up x_F\; \Up x_G}
       &\text{ if $H = F\cup j = G\cup i\in\Sigma$,}\\\\
    \sum\limits_{j\colon F\cup j\in\Sigma} \dfrac{X_{F\cup j}}{\Up X_F}
       &\text{ if $F=G$,}\\\\
    \quad0 &\text{ otherwise.}
  \end{cases}
\end{equation}
Moreover,
\begin{equation}\label{h-detN}
 \hat{\tau}_d(\Sigma) =   
  \det \Lwud_U = \left(\prod_{F\in\Lambda_{d-1}} \Up X_F\right) \det N.
\end{equation}

We will shortly see (Lemma~\ref{reduce-to-evals}) that $N$ is almost identical to the (full) Laplacian of
the deletion $\Delta=\del_p\Sigma$.

\subsection{The algebraic fine weighting} \label{alg-fine-weighting-section}

The combinatorial fine weighting just defined is awkward to work with directly,
because the simplicial boundary and coboundary maps
\eqref{fine-boundary} and \eqref{fine-coboundary} do not fit together into
an algebraic chain complex.  Therefore, we introduce a new weighting by Laurent
monomials, the \emph{algebraic fine weighting}, that
does give the structure of a
chain complex, behaves well with respect to cones and near-cones, 
and is easy to translate into the combinatorial fine
weighting.

\begin{definition}
Let $\Delta^d$ be a simplicial complex on vertices $V\subset\Nn$.
The \emph{algebraic finely weighted simplicial boundary map} of $\Delta$
is the homomorphism $\Abd_{\Delta,i}:C_i(\Delta)\to C_{i-1}(\Delta)$ given by
  \begin{equation} \label{vertex-boundary}
  \Abd_{\Delta,i}[ F] = \sum_{j\in F} \sign(j, F)
    \frac{\Up^{d-i}(x_F)}{\Up^{d-i+1}(x_{ F\sm j})} [ F\sm j]
  \end{equation}
and similarly the \emph{algebraic finely weighted simplicial coboundary map}
$\Acbd_{\Delta,i+1}:C_i(\Delta)\to C_{i+1}(\Delta)$ is given by
  \begin{equation} \label{vertex-coboundary}
  \Acbd_{\Delta,i+1}[ F] = \sum_{j\in V\sm F} \sign(j, F\cup j)
    \frac{\Up^{d-i-1}(x_{ F\cup j})}{\Up^{d-i}(x_F)} [ F\cup j].
  \end{equation}
\end{definition}
We will sometimes drop one or both subscripts when no confusion can arise.
By the formula~\eqref{Up-operator}, we can apply the raising operator $\Up$
to $\Abd$ and $\Acbd$ by applying it to each matrix entry.  That is,
for $[F]\in C_i(\Delta)$ and $a\in\Nn$, we have
  \begin{align}
  \Up^a\Abd_{\Delta,i}[ F] &= \sum_{j\in F} \sign(j, F)
    \frac{\Up^{d-i+a}(x_F)}{\Up^{d-i+a+1}(x_{ F\sm j})} [ F\sm j],
  \label{raised-vertex-boundary}\\
  \Up^a\Acbd_{\Delta,i}[ F] &= \sum_{j\in V\sm F} \sign(j, F\cup j)
    \frac{\Up^{d-i+a-1}(x_{ F\cup j})}{\Up^{d-i+a}(x_F)} [ F\cup j].
  \label{raised-vertex-coboundary}
  \end{align}

\begin{lemma} \label{boundary-map-lemma}
Let $a\in\Nn$.  Then
$\Up^a\Abd \, \circ\Up^a\Abd=0$ and $\Up^a\Acbd\circ\Up^a\Acbd=0$.
That is, the algebraic finely weighted boundary and coboundary operators induce chain complexes
  $$\cdots \rightarrow C_{i+1}(\Delta) \xrightarrow{\Abd_{\Delta,i+1}} C_i(\Delta)
    \xrightarrow{\Abd_{\Delta,i}} C_{i-1}(\Delta) \rightarrow \cdots$$
  $$\cdots \leftarrow  C_{i+1}(\Delta) \xleftarrow{\Acbd_{\Delta,i+1}} C_i(\Delta)
    \xleftarrow{\Acbd_{\Delta,i}} C_{i-1}(\Delta) \leftarrow \cdots$$
\end{lemma}

\begin{proof}
Since the matrices that represent the maps $\Abd$ and $\Acbd$ are mutual transposes,
it suffices to prove the first assertion.  For $ F\in\Delta_i$, we have by
\eqref{raised-vertex-boundary}
  \begin{align*}
  &\Up^a\Abd_{\Delta,i-1}(\Up^a\Abd_{\Delta,i}([F]))
  ~=~ \Up^a\Abd_{\Delta,i-1} \left( \sum_{j\in F} \sign(j, F)
    \frac{\Up^{d-i+a}(x_F)}{\Up^{d-i+a+1}(x_{F\sm j})} [F\sm j] \right)\\
  &=~ \sum_{j\in F} \sign(j, F) \frac{\Up^{d-i+a}(x_F)}{\Up^{d-i+a+1}(x_{F\sm j})}
    \cdot\Up^a\Abd_{\Delta,i-1}([F\sm j])\\
  &=~ \sum_{j\in F} \sign(j, F)
       \frac{\Up^{d-i+a}(x_F)}{\Up^{d-i+a+1}(x_{F\sm j})}
       \sum_{k\in F\sm j} \sign(k, F\sm j)
         \frac{\Up^{d-i+a+1}(x_{F\sm j})}{\Up^{d-i+a+2}(x_{F\sm j\sm k})} [F\sm j\sm k]\\
  &=~ \sum_{j\in F} \sum_{k\in F\sm j} \sign(j, F) \sign(k, F\sm j)
       \frac{\Up^{d-i+a}(x_F)}{\Up^{d-i+a+2}(x_{F\sm j\sm k})}
          [F\sm j\sm k]\\
  &=~ \sum_{\substack{j,k\in F\\j \neq k}}
        \Big(\sign(j,F) \sign(k,F\sm j)
        + \sign(k,F) \sign(j,F\sm k) \Big)
       \frac{\Up^{d-i+a}(x_F)}{\Up^{d-i+a+2}(x_{F\sm j\sm k})}
          [F\sm j\sm k]
  \end{align*}
and it is a standard fact of simplicial homology theory that the parenthesized expression is zero.
\end{proof}

Define the algebraic finely weighted up-down, down-up, and total Laplacians by
  $$
  \LLud_{\Delta,i} = \Abd_{\Delta,i+1}\Acbd_{\Delta,i+1}, \qquad
  \LLdu_{\Delta,i} = \Acbd_{\Delta,i}\Abd_{\Delta,i}, \qquad
  \LLtot_{\Delta,i} = \LLud_{\Delta,i} + \LLdu_{\Delta,i}.
  $$
Each of these is a linear endomorphism of $C_i(\Delta)$,
represented by a symmetric matrix, hence diagonalizable.
Let $\Sud_i(\Delta)$, $\Sdu_i(\Delta)$,
and $\Stot_i(\Delta)$ denote the spectra (multisets of eigenvalues)
of $\LLud_{\Delta,i}$, $\LLdu_{\Delta,i}$, and $\LLtot_{\Delta,i}$
respectively.  We will use the abbreviations $\LLud$, $\LLdu$, $\LLtot$,
$\Sud$, $\Sdu$, $\Stot$ when no confusion can arise.

If we regard $\LLud_{\Delta,d-1}$ as a matrix with rows and columns indexed by $\Delta_{d-1}$, 
then it is not hard to check that its $(F,G)$ entry is
\begin{equation}\label{LLud-entries}
(\LLud_{\Delta,d-1})_{FG} = 
  \begin{cases}
    \sign(j,H)\ \sign(i,H)\ \dfrac{X_H}{\Up x_F\;\Up x_G}
       &\text{ if $H = F\cup j = G\cup i\in\Delta$,}\\\\
    \sum\limits_{j\colon F\cup j \in\Delta} \dfrac{X_{F\cup j}}{\Up X_F}
       &\text{ if $F=G$,}\\\\
    \quad0 &\text{ otherwise.}
  \end{cases}
\end{equation}

This matrix is almost identical to the matrix $N(\Sigma)$ defined in \eqref{N-entries}
when $\Delta = \del_p \Sigma$, as we now explain.

\begin{lemma}\label{reduce-to-evals}
Let $\Sigma^d$ be a pure shifted complex with initial vertex $p$, and let
$\Lambda = \link_p \Sigma$ and $\Delta = \del_p \Sigma$.  Then
$$
 \hat{\tau}_d(\Sigma)
  = \left(\prod_{F\in\Lambda_{d-1}} \Up X_F\right) 
     \prod_{\lambda \in \Sud_{\Delta,d-1}} (X_{1,p}+\lambda).
$$
\end{lemma}

\begin{proof}
First, note that $\LL=\LLud_{\Delta,d-1}$ is indexed by the faces of $\Delta_{d-1}$ 
and $N=N(\Sigma)$
is indexed by the faces of $\Lambda_{d-1}$.  These indexing sets coincide
because $\Sigma$ is shifted and pure of dimension~$d$.

Second, we show that the off-diagonal entries (the first cases in
\eqref{N-entries} and \eqref{LLud-entries}) coincide.  Suppose that
$F,G$ are distinct faces in $\Delta_{d-1}=\Lambda_{d-1}$,
and that $H=F\cup i=G\cup j$.
Suppose that $i\neq j$ and $H=F\cup i=G\cup j$.  We must show that
$H\in\Delta$ if and only if $H\in\Sigma$.  The ``only if''
direction is immediate because $\Delta\subset\Sigma$.
On the other hand, $H=F\cup G$ and no element of $\Delta_{d-1}$ contains~$p$.
Therefore, if $H\in\Sigma$, then $H\in\Delta$, as desired.

Third, we compare the entries on the main diagonals of $\LL$
and $N$.  Their only difference is that the summand with $j=p$ occurs in
the second case of \eqref{N-entries}, but not in \eqref{LLud-entries}.  Hence
  \begin{equation}\label{N-diagonal}
  N_{FF} ~=~ \LL_{FF} + \frac{X_{F\cup p}}{\Up X_F}
  ~=~ \LL_{FF} + X_{1,p}
  \end{equation}
by~\eqref{tilde:monom}.  Therefore $N = \LL + X_{1,p}I$,
where $I$ is an identity matrix of size $f_{d-1}(\Delta)$, and
  $$\det N = \chi(-\LL,X_{1,p}) 
           = \prod_{\lambda \in \Sud_{\Delta,d-1}} (X_{1,p}+\lambda),$$
where $\chi$ denotes the characteristic polynomial of $-\LL$ in the variable $X_{1,p}$.
The lemma now follows from equation~\eqref{h-detN}.
\end{proof}

The goal of the next two sections is to compute $\LLud_{\Delta,d-1}$.

\section{Cones and near-cones}
\label{cone-section}

A shifted complex is an iterated near-cone, so we want to describe
the Laplacian eigenvalues of a near-cone in terms of its base.
Before we do so, we must consider the case of a cone.
Proposition~\ref{cone-spectrum} provides the desired recurrence
for cones, and Proposition~\ref{near-cone-recurrence} for near-cones.

\subsection{Boundary and coboundary operators of cones}
\label{section:cone-boundary}

Let $\Gamma$ be the simplicial complex with the single vertex 1,
and let $\Delta^d$ be any complex on $V=[2,n]$.
For a face $F\in\Delta$, write $\tilde F=1\cup F$.
The corresponding \emph{cone} is
  $$\Sigma = 1*\Delta = \Gamma*\Delta = \{ F,\tilde F \st  F\in\Delta\}.$$

We will make use of the identification
$$C_i(\Sigma) ~\isom~ (C_{-1}(\Gamma)\tensor C_i(\Delta)) \dsum (C_0(\Gamma)\tensor C_{i-1}(\Delta)).$$

By \eqref{raised-vertex-boundary} and \eqref{raised-vertex-coboundary},
the (raised) boundary and coboundary maps on $\Gamma$ are given explicitly by
  \begin{align*}
  \Up^a \Abd_{\Gamma}[1]  &= x_{a+1,1} [\0], & \Up^a \Acbd_{\Gamma}[1]  &=0,\\
  \Up^a \Abd_{\Gamma}[\0] &= 0,              & \Up^a \Acbd_{\Gamma}[\0] &= x_{a+1,1} [1].
  \end{align*}

Next, we give explicit formulas for the maps $\Abd_{\Sigma,i}$ and $\Acbd_{\Sigma,i+1}$. 
How these maps act on a face of $\Sigma$
depends on whether it is of the form $ F$, for $ F\in\Delta_i$, or
$\tilde F$, for $ F\in\Delta_{i-1}$.  Note that in any case $\sign(1,\tilde F)=1$,
and that for all $v\in F$ we have $\sign(v, F) = -\sign(v,\tilde F)$.  Therefore,

\begin{subequations}
\begin{align}
\Abd_{\Sigma,i} & ([\0]\tensor[F])
  = \sum_{j\in F} \sign(j, F)
    \frac{\Up^{d-i+1}(x_F)}{\Up^{d-i+2}(x_{ F\sm j})}
    [\0]\tensor[ F\sm j]\notag\\
 &= \Up\left(\sum_{j\in F} \sign(j, F)
    \frac{\Up^{d-i}(x_F)}{\Up^{d-i+1}(x_{ F\sm j})}
    [\0]\tensor[ F\sm j]\right)\notag\\
 &= (\id\tensor\Up\Abd_{\Delta,i})([\0]\tensor[ F]),
\label{del-no-apex}
\end{align}

\begin{align}
\Abd_{\Sigma,i} & ([1]\tensor[ F])
  = \sign(1,\tilde F)
    \frac{\Up^{d-i+1}(x_{\tilde F})}{\Up^{d-i+2}(x_F)}
    [\0]\tensor[ F]
  + \sum_{j\in F}
      \sign(j,\tilde F)
      \frac{\Up^{d-i+1}(x_{\tilde F})}{\Up^{d-i+2}(x_{\tilde F\sm j})}
      [1]\tensor[ F\sm j]\notag\\
 &= x_{d-i+2,1} [\0]\tensor[ F]
  - \sum_{j\in F}
      \sign(j, F)
      \frac{x_{d-i+2,1}}{x_{d-i+3,1}}
      \frac{\Up^{d-i+2}(x_F)}{\Up^{d-i+3}(x_{ F\sm j})}
      [1]\tensor[ F\sm j]\notag\\
 &= x_{d-i+2,1} [\0]\tensor[ F]
  - \frac{x_{d-i+2,1}}{x_{d-i+3,1}} \Up\left(
      \sum_{j\in F} \sign(j, F)
      \frac{\Up^{d-i+1}(x_F)}{\Up^{d-i+2}(x_{ F\sm j})}
      [1]\tensor[ F\sm j]\right) \notag\\
 &= \left( \Up^{d-i+1}\Abd_{\Gamma}\tensor\id
        - \frac{x_{d-i+2,1}}{x_{d-i+3,1}}\id\tensor\Up\Abd_{\Delta,i-1} \right) ([1]\tensor[ F]),
\label{del-apex}
\end{align}

\begin{align}
\Acbd_{\Sigma,i+1} & ([\0]\tensor[ F])
  = \sign(1,\tilde F)
    \frac{\Up^{d-i}(x_{\tilde F})}{\Up^{d-i+1}(x_F)}
    [1]\tensor[ F]
  + \sum_{j\in V\sm F}\!\!
    \sign(j, F\cup j)
    \frac{\Up^{d-i}(x_{ F\cup j})}{\Up^{d-i+1}(x_F)}
    [\0]\tensor[ F\cup j]\notag\\
 &= x_{d-i+1,1} [1]\tensor[ F]
  + \Up\left( \sum_{j\in V\sm F}
    \sign(j, F\cup j)
    \frac{\Up^{d-i-1}(x_{ F\cup j})}{\Up^{d-i}(x_F)}
    [\0]\tensor[ F\cup j] \right)\notag\\
 &= \Big( \Up^{d-i}\Acbd_{\Gamma}\tensor\id + \id\tensor\Up\Acbd_{\Delta,i+1} \Big)
    ([\0]\tensor[ F]),
\label{codel-no-apex}
\end{align}

\begin{align}
\Acbd_{\Sigma,i+1} & ([1]\tensor[ F])
  = \sum_{j\in V\sm F}
    \sign(j,\tilde F\cup j)
    \frac{\Up^{d-i}(x_{\tilde F\cup j})}{\Up^{d-i+1}(x_{\tilde F\cup j})}
    [1]\tensor[ F\cup j] \notag\\
 &= -\sum_{j\in V\sm F}
    \sign(j, F\cup j)
    \frac{\Up^{d-i}(x_{1,1})}{\Up^{d-i+1}(x_{1,1})}
    \frac{\Up^{d-i+1}(x_{ F\cup j})}{\Up^{d-i+2}(x_F)}
    [1]\tensor[ F\cup j] \notag\\
 &= - \frac{x_{d-i+1,1}}{x_{d-i+2,1}} \cdot \Up\left(
    \sum_{j\in V\sm F} \sign(j, F\cup j)
    \frac{\Up^{d-i}(x_{ F\cup j})}{\Up^{d-i+1}(x_F)}
    [1]\tensor[ F\cup j]\right) \notag\\
 &= \left( -\frac{x_{d-i+1,1}}{x_{d-i+2,1}}\id\tensor \Up\Acbd_{\Delta,i} \right)
    ([1]\tensor[ F]).
\label{codel-apex}
\end{align}
\end{subequations}

\subsection{Eigenvectors of cones}

In order to describe the Laplacian eigenvalues and eigenvectors of $1*\Delta$ in terms
of those of $\Delta$, we first need some basic facts about the Laplacians
of an arbitrary simplicial complex.  The following proposition does not depend on fine
weighting, and works with any weighted boundary map that satisfies $\partial^2=0$.

\begin{proposition} \label{Laplacian-facts}
Let $\Omega^d$ be a simplicial complex, and let $-1\leq i\leq d$.
Let $\LLud_i = \LLud_{\Omega,i}$, $\LLdu_i = \LLdu_{\Omega,i}$, 
$\Abd_i = \Abd_{\Omega,i}$, and $\Acbd_i = \Acbd_{\Omega,i}$.

\begin{enumerate}

\item\label{decompose-chain}
 The chain group $C_i(\Omega)$ decomposes as a direct sum
  \begin{equation} \label{eigenspace-decomposition}
  C_i(\Omega) = C^{\rm ud}_i(\Omega) \dsum C^{\rm du}_i(\Omega) \dsum C^0_i(\Omega)
  \end{equation}
where
  \begin{itemize}
    \item $C^{\rm ud}_i(\Omega)$ has a basis consisting of eigenvectors of
      $\LLud_i$ whose eigenvalues are all nonzero, and on which
      $\LLdu_i$ acts by zero;
    \item $C^{\rm du}_i(\Omega)$ has a basis consisting of eigenvectors of
      $\LLdu_i$ whose eigenvalues are all nonzero, and on which
      $\LLud_i$ acts by zero; and
    \item $\LLud_i$ and $\LLdu_i$ both act by zero on $C^0_i(\Omega)$.
  \end{itemize}

\item\label{equal-kernels}
 $\ker(\LLud_i)=\ker(\Acbd_{i+1})$ and $\ker(\LLdu_i)=\ker(\Abd_i)$.

\item\label{homology-count}
$\dim C^0_i(\Omega)=\betti_i(\Omega)$, the $i^{th}$ Betti number of $\Omega$.

\item\label{spectra-identities}
 Each of the spectra $\Sud_i$, $\Sdu_i$, $\Stot_i$ of $\Omega$ has cardinality $f_i(\Omega)$ as a multiset, and
  \begin{subequations}
  \begin{align}
  \Sdu_i &\zeq \Sud_{i-1}, \ \text{ and} \label{UD-DU}\\
  \Stot_i &\zeq \Sud_i \cup \Sdu_i \zeq \Sud_i \cup \Sud_{i-1}. \label{total-spectrum}
  \end{align}
  \end{subequations}

\end{enumerate}
\end{proposition}

\begin{proof}
For assertion (1), note that $\LLud_i\LLdu_i=\Abd\Acbd\Acbd\Abd=0$ and
$\LLdu_i\LLud_i=\Acbd\Abd\Abd\Acbd=0$.  Thus we may simply take $C^{\rm ud}_i(\Omega)$
and $C^{\rm du}_i(\Omega)$ to be the spans of the eigenvectors of $\LLud_i$
and $\LLdu_i$ with nonzero eigenvalues.

For assertion (2),
the operators $\LLud_i$ and $\Acbd_{i+1}$ act on the same space,
namely $C_i(\Omega)$, and they have the same rank (this is just the linear
algebra fact that $\rank(MM^T)=\rank M$ for any matrix $M$).  Therefore,
their kernels have the same dimension.  Clearly $\ker(\LLud_i)\supseteq\ker(\Acbd_{i+1})$,
so we must have equality.  The same argument shows that $\ker(\LLdu_i)=\ker(\Abd_i)$.

Assertion (3) follows from the calculation
\begin{align*}
\betti_i(\Omega) ~=~ \dim\HH_i(\Omega,\Qq)
  &= (\dim\,\ker\Abd_i)-(\dim\,\im\Abd_{i+1})\\
  &= (f_i-\rank\Abd_i)-(\rank\Abd_{i+1})\\
  &= \dim C_i - \dim C^{\rm ud}_i - \dim C^{\rm du}_i ~=~ \dim C^0_i.
\end{align*}

For assertion (4), first note that $\dim C_i(\Omega)=f_i(\Omega)$, and that
the matrix representing each spectrum is symmetric, hence diagonalizable.
The identity \eqref{UD-DU} is a standard fact in linear algebra,
and \eqref{total-spectrum} is a consequence of the decomposition
\eqref{eigenspace-decomposition}.
\end{proof}

\begin{proposition} \label{cone-spectrum}
As in Section~\ref{section:cone-boundary}, let $\Delta^d$ be a simplicial
complex on vertex set~$[2,n]$, and let $\Sigma=1*\Delta$.  Then
\begin{equation} \label{cone-updown-spectrum}
\begin{aligned}
\Sud_i(\Sigma)
  &\zeq \left\{ X_{d-i+1,1}+\Up\lambda \st
    \lambda\in\Sud_i(\Delta),\ \lambda\neq 0 \right\}\\
  &\qquad\cup \left\{ X_{d-i+1,1}+\frac{X_{d-i+1,1}}{X_{d-i+2,1}} \Up\mu \st
    \mu\in\Sud_{i-1}(\Delta),\ \mu\neq 0 \right\}\\
  &\qquad\cup \left\{ X_{d-i+1,1} \right\}^{\betti_i(\Delta)}.
\end{aligned}
\end{equation}
\end{proposition}

Here the symbol $\cup$ denotes multiset union, and the superscript in the last line
indicates multiplicity.

\begin{proof}
Throughout the proof, we abbreviate $\LLud_{\Sigma,i}$ by
$\mathbf{L}$.  All other Laplacians that arise will be specified precisely.

First, let $V\in C^{\rm ud}_i(\Delta)$ be an eigenvector of
$\LLud_{\Delta,i}$ with eigenvalue $\lambda\neq 0$.  Then
$\Up \LLud_{\Delta,i} (\Up V) = \Up\lambda \Up V$,
and, by Lemma~\ref{boundary-map-lemma},
\begin{equation} \label{dirty-trick}
\Up\Abd_{\Delta,i}(\Up V)
  = \frac{1}{\Up\lambda} \Up\Abd_{\Delta,i} \big(\Up\Abd_{\Delta,i+1} (\Up\Acbd_{\Delta,i+1} (\Up V))\big)
  = 0.
\end{equation}

Using \eqref{del-no-apex}\dots\eqref{codel-apex} and \eqref{dirty-trick},
we calculate 
  \begin{align}
  \mathbf{L} \left([\0] \tensor \Up V\right)
    &= \Abd_{\Sigma,i} ( \Acbd_{\Sigma,i} ( [\0]\tensor \Up V ))\notag\\
    &= \Abd_{\Sigma,i} \left( \Up^{d-i}\Acbd_{\Gamma}[\0] \tensor \Up V 
      + [\0] \tensor \Up\Acbd_{\Delta,i+1}(\Up V) \right)\notag\\
    &= x_{d-i+1,1} \Abd_{\Sigma,i+1} \left( [1] \tensor \Up V \right)
      + \Abd_{\Sigma,i+1} \left( [\0] \tensor \Up\Acbd_{\Delta,i+1}(\Up V) \right)\notag\\
    &= x_{d-i+1,1} \left( \Up^{d-i}\Abd_{\Gamma} [1] \tensor \Up V
      - \frac{x_{d-i+1,1}}{x_{d-i+2,1}} [1]\tensor \Up\Abd_{\Delta,i}(\Up V) \right)\notag\\
    &\quad+ [\0] \tensor \Up\Abd_{\Delta,i+1}(\Up\Acbd_{\Delta,i+1}(\Up V))\notag\\
    &= (X_{d-i+1,1} + \Up\lambda) \left([\0] \tensor \Up V\right).
    \label{UD-eigenvector}
  \end{align}
  Therefore, $[\0] \tensor \Up V \in C_i(\Sigma)$
is an eigenvector of $\mathbf{L}$, with eigenvalue $X_{d-i+1,1} + \Up\lambda$.  
Notice that this eigenvalue cannot be zero:
since $1\not\in\Delta$, no Laplacian eigenvalue of $\Delta$ can
possibly equal $-X_{d-i,1}$.

Second, let $W$ be an eigenvector in $C^{\rm du}_i(\Delta)$ with nonzero
eigenvalue~$\mu$.  That is,
$\LLdu_{\Delta_i}(W) = \Acbd_{\Delta,i}(\Abd_{\Delta,i}(W)) = \mu W$,
and $\Acbd_{\Delta,i+1}(W)=0$
by a computation similar to \eqref{dirty-trick}.
Define
  $$A=[\0]\tensor\Up W,\qquad B=[1]\tensor \Up\Abd_{\Delta,i}(\Up W).$$
Note that these are both nonzero elements of $C_i(\Sigma)$.  Then
\begin{subequations}
\begin{align}
  \mathbf{L}(A) &= \Abd_{\Sigma,i+1} \left( \Acbd_{\Sigma,i+1} ( [\0]\tensor\Up W) \right)\notag\\
     &= \Abd_{\Sigma,i+1} \left( \Up^{d-i}\Acbd_{\Gamma}[\0]\tensor\Up W +
        [\0]\tensor\Up\Acbd_{\Delta,i+1}(\Up W) \right)\notag\\
     &= x_{d-i+1,1} \Abd_{\Sigma,i+1} \left([1]\tensor\Up W\right) \label{clubs}\\
     &= x_{d-i+1,1} \left( \Up^{d-i}\Abd_{\Gamma}[1]\tensor\Up W -
	\frac{x_{d-i+1,1}}{x_{d-i+2,1}}[1]\tensor\Up\Abd_{\Delta,i}(\Up W)\right)\notag\\
     &= X_{d-i+1,1} A - \left(\frac{X_{d-i+1,1}}{x_{d-i+2,1}}\right) B\label{spades}
\end{align}
\end{subequations}
and
\begin{align*}
 \mathbf{L}(B) &= \Abd_{\Sigma,i+1} \left( \Acbd_{\Sigma,i+1} ([1]\tensor \Up\Abd_{\Delta,i}(\Up W)) \right)\\
     &= \Abd_{\Sigma,i+1} \left( -\frac{x_{d-i+1,1}}{x_{d-i+2,1}}[1]\tensor
          \Up\Acbd_{\Delta,i}( \Up\Abd_{\Delta,i}( \Up W)) \right)\\
     &= -\frac{x_{d-i+1,1}}{x_{d-i+2,1}} \Up\mu \cdot
          \Abd_{\Sigma,i+1} \left( [1]\tensor \Up W \right)\\
     &= \left(-\frac{X_{d-i+1,1}}{x_{d-i+2,1}} \Up\mu\right) A +
        \left( \frac{X_{d-i+1,1}}{X_{d-i+2,1}} \Up\mu\right) B,
\end{align*}
where the last step follows by the equality of \eqref{clubs} and \eqref{spades}.
Letting
  $$f = X_{d-i+1,1}, \qquad g = -\frac{X_{d-i+1,1}}{x_{d-i+2,1}}, \qquad
    h = -\frac{\Up\mu}{x_{d-i+2,1}},$$
 the calculations above say that
  $$\mathbf{L}(A) = fA+gB, \qquad \mathbf{L}(B) = h(fA+gB),$$
which implies that
  $$\mathbf{L}(fA+gB) = f\mathbf{L}(A)+g\mathbf{L}(B) = f(fA+gB)+gh(fA+gB) = (f+gh)(fA+gB).$$
That is, $fA+gB$ is an eigenvector of $\mathbf{L}$ with eigenvalue
  $$f+gh = X_{d-i+1,1} + \frac{X_{d-i+1,1}}{X_{d-i+2,1}}\Up\mu.$$
As before, this quantity cannot be zero because $\Delta$ does not contain
the vertex~1.

Third, let $Z\in C_i^0(\Delta)$; we will show that
$[\0]\tensor \Up Z$ is an eigenvector of $\mathbf{L}$ with eigenvalue $X_{d-i+1,1}$.
Indeed, by (\ref{equal-kernels}) of Proposition~\ref{Laplacian-facts}, we have
$\Abd_{\Delta,i}(Z) = \Acbd_{\Delta,i+1}(Z) = 0$, so that
  $$
  \Up\Abd_{\Delta,i}(\Up Z) = \Up\Acbd_{\Delta,i+1}(\Up Z) = 0.
  $$
Therefore
  \begin{align}
  \mathbf{L}([\0]\tensor \Up Z)
  &= \Abd_{\Sigma,i+1}(\Acbd_{\Sigma,i+1}([\0]\tensor \Up Z))\notag\\
  &= \Abd_{\Sigma,i+1}\left( (\Up^{d-i}\Acbd_{\Gamma}[\0]\tensor \Up Z)
    + [\0]\tensor \Up\Acbd_{\Delta,i+1}(\Up Z) \right)\notag\\
  &= x_{d-i+1,1} \Abd_{\Sigma,i+1} ([1]\tensor \Up Z)\notag\\
  &= x_{d-i+1,1}\cdot\left( \Up^{d-i}\Abd_{\Gamma}[1]\tensor\Up Z
     - \frac{x_{d-i+1,1}}{x_{d-i+2}}[1]\tensor\Up\Abd_{\Delta,i}(\Up Z) \right)\notag\\
  &= X_{d-i+1,1}([\0]\tensor\Up Z),
    \label{UD-nullvector}
  \end{align}
as desired.  By assertion~(\ref{homology-count}) of
Proposition~\ref{Laplacian-facts}, the multiplicity of this eigenvalue is
$\dim C_i^0(\Delta)=\betti_i(\Delta)$.

At this point, we have proven that \eqref{cone-updown-spectrum} is true if ``$\zeq$''
is replaced with ``$\supseteq$''.  On the other hand, we have accounted for
  $$\dim C^{\rm ud}_i(\Delta) + \dim C^{\rm du}_i(\Delta) + \dim C^0_i(\Delta) = \dim C_i(\Delta) = f_i(\Delta)$$
nonzero eigenvalues in $\Sud_i(\Sigma)$.  Since the preceding calculations
hold for all $i$, we also know $f_{i-1}(\Delta)$ nonzero eigenvalues
of $\LLud_{\Sigma,i-1}$ ($\zeq \LLdu_{\Sigma,i}$).
By \eqref{total-spectrum}, we have accounted for all
$f_i(\Delta)+f_{i-1}(\Delta)=f_i(\Sigma)=\dim C_i(\Sigma)$
eigenvalues of $\LLtot_{\Sigma,i}$.  So we have indeed found all
the nonzero eigenvalues of~$\mathbf{L}$.
\end{proof}

One can obtain explicit formulas for the spectra $\Sdu_i(\Sigma)$ and $\Stot_i(\Sigma)$
by applying \eqref{UD-DU} and \eqref{total-spectrum} to the formula
\eqref{cone-updown-spectrum}; we omit the details.

\subsection{Eigenvalues of near-cones}

The next step is to establish a recurrence (Proposition~\ref{near-cone-recurrence})
for the Laplacian eigenvalues of near-cones, in terms of the eigenvalues of
the deletion and the link of the apex.  Our method is based on that of
Lemma~5.3 of \cite{DR}.
By itself, Proposition~\ref{near-cone-recurrence} is not a proper recurrence,
in the sense that it computes the eigenvalues of a pure complex in terms
of complexes that are not necessarily pure.  Therefore, it cannot be applied recursively;
the proper recurrence for shifted complexes will have to wait until
Theorem~\ref{shifted-recurrence}.  Since we will be comparing complexes with
similar face sets but of different dimensions, we begin by describing how their
Laplacian spectra are related.  

\begin{lemma} \label{skeleton-Laplacian}
Let $\Sigma^d$ be a simplicial complex, and let $j<i\leq d$.  Then
  $\Sud_j(\Sigma) = \Up^{d-i} \Sud_j(\Sigma_{(i)}).$
\end{lemma}

\begin{proof}
The complexes $\Sigma_{(i)}$ and $\Sigma$ have the same face sets for every
dimension $\leq i$, but $\dim\Sigma_{(i)} = \dim\Sigma - (d-i)$.  Therefore,
the algebraic finely weighted boundary maps and Laplacians of $\Sigma$ can be obtained
from those of $\Sigma_{(i)}$ by applying $\Up^{d-i}$, from which the lemma follows.
\end{proof}

Recall that the \emph{pure $i$-skeleton} of $\Sigma$ is the subcomplex
$\Sigma_{[i]}$ generated by the $i$-dimensional faces of $\Sigma$.

\begin{lemma}\label{pure-top-Sud}
Let $\Sigma^d$ be a simplicial complex.  Then
$
\Sud_{d-1}(\Sigma) \zeq \Sud_{d-1}(\Sigma_{[d]}).
$
\end{lemma}

\begin{proof}
This result is proved in~\cite[Lemma 3.2]{Duval}, but we sketch the proof
here for completeness.  
First, observe that $\LLud_{d-1}$ depends only on
$(d-1)$- and $d$-dimensional faces.  
Letting $\Omega=\Sigma_{[d]}$, we have $\Omega_d=\Sigma_d$
and $\Omega_{d-1}\subseteq\Sigma_{d-1}$, and indeed
$\LLud_{\Sigma,d-1}[F] = \LLud_{\Omega,d-1}[F]$
for any $F\in\Omega_{d-1}$.  On the other hand,
$\Sigma_{d-1}\sm\Omega_{d-1}$ consists precisely of those
faces $G$ not contained in any $d$-dimensional faces of $\Sigma$.
But $\LLud_{\Sigma,d-1}$ acts by zero on any such $G$, and
the lemma follows immediately.
\end{proof}

\begin{corollary}\label{pure-Sud}
Let $\Sigma^d$ be a simplicial complex.  Then
$
\Sud_{i-1}(\Sigma) ~\zeq~ \Up^{d-i}\Sud_{i-1}(\Sigma_{[i]}).
$
\end{corollary}
\begin{proof}
We have $\Sud_{i-1}(\Sigma) ~=~ \Up^{d-i}\Sud_{i-1}(\Sigma_{(i)}) 
~\zeq~ \Up^{d-i}\Sud_{i-1}((\Sigma_{(i)})_{[i]})
~=~ \Up^{d-i}\Sud_{i-1}(\Sigma_{[i]})$
by Lemmas~\ref{skeleton-Laplacian} and~\ref{pure-top-Sud}.
(Note that $(\Sigma_{(i)})_{[i]} = \Sigma_{[i]}$.)
\end{proof}

\begin{proposition} \label{near-cone-recurrence}
Let $\Sigma^d$ be a pure near-cone with apex $p$, and let 
$\Delta = \del_{p}\Sigma$ and $\Lambda = \link_{p}\Sigma$ be the deletion and link, 
respectively, of $\Sigma$ with respect to vertex $p$.  Then
  \begin{equation*}
  \begin{aligned}
  \Sud_{d-1}(\Sigma)
    &\zeq \left\{ X_{1,p}+\lambda \st
     \lambda\in\Sud_{d-1}(\Delta),\ \lambda\neq 0 \right\}\\
    &\qquad\cup \left\{ X_{1,p}+\frac{X_{1,p}}{X_{2,p}} \Up\mu \st
     \mu\in\Sud_{d-2}(\Lambda),\ \mu\neq 0 \right\}\\
    &\qquad\cup \left\{ X_{1,p} \right\}^{\betti_{d-1}(\Delta)}.
  \end{aligned}
  \end{equation*}
\end{proposition}

\begin{proof}
If $\Sigma$ is a cone, then the result follows from direct application of Proposition~\ref{cone-spectrum}.  
(Note that $\Sud_{d-1}(\Delta)$ consists of only $0$'s in this case, since $\Delta$ is only $(d-1)$-dimensional.)
Thus, we may as  well assume for the remainder of the proof that $\Sigma$ is not a cone. 
In this case,  $\dim \Delta = d$ and $\dim \Lambda = d-1$.  
It is not difficult to see that
\begin{align}
\label{Delta-Lambda} \Delta_{(d-1)} &= \Lambda, \, \, \textrm{and} \\
\label{Sigma-Delta} \Sigma_{(d)} = \Sigma &= (p * \Delta)_{(d)}.
\end{align}
Applying Lemma~\ref{skeleton-Laplacian} to equation~\eqref{Sigma-Delta}
(keeping in mind that $\dim(p * \Delta) =  d+1$), and then applying
Proposition~\ref{cone-spectrum}, we find that
\begin{align*}
\Up\Sud_{d-1}(\Sigma)
  &= \Sud_{d-1}(p*\Delta)\\
  &\zeq \left\{ X_{2,p}+\Up\lambda \st
   \lambda\in\Sud_{d-1}(\Delta),\ \lambda\neq 0 \right\} \\
  &\qquad\cup\left\{ X_{2,p}+\frac{X_{2,p}}{X_{3,p}} \Up\mu \st
   \mu\in\Sud_{d-2}(\Delta),\ \mu\neq 0 \right\}\notag\\
  &\qquad\cup\left\{ X_{2,p} \right\}^{\betti_{d-1}(\Delta)}
\end{align*}
so that
\begin{align*}
\Sud_{d-1}(\Sigma)
  &= \left\{ X_{1,p}+\lambda \st
   \lambda\in\Sud_{d-1}(\Delta),\ \lambda\neq 0 \right\} \\
  &\qquad\cup \left\{ X_{1,p}+\frac{X_{1,p}}{X_{2,p}} \mu \st
   \mu\in\Sud_{d-2}(\Delta),\ \mu\neq 0 \right\} \\
  &\qquad\cup \left\{ X_{1,p} \right\}^{\betti_{d-1}(\Delta)}. 
\end{align*}
The desired result now follows from applying Lemma~\ref{skeleton-Laplacian}
to equation~\eqref{Delta-Lambda}.
\end{proof}

\section{The Laplacian spectrum of a shifted complex}
\label{z-poly-section}

In this section, we explicitly describe the eigenvalues of the
algebraic finely weighted Laplacians of a shifted complex~$\Sigma$.
The eigenvalues are Laurent polynomials
called \emph{$z$-polynomials}, which are in one-to-one correspondence
with the \emph{critical pairs} of the complex: pairs $(A,B)$ such that $A\in\Sigma$,
$B\not\in\Sigma$, and $B$ covers $A$ in the componentwise partial order.
The main result, Theorem~\ref{zpoly-theorem}, is proved by
establishing 
identical recurrences for the $z$-polynomials (Theorem~\ref{shifted-recurrence}) 
and critical pairs (Corollary~\ref{allsg-shifted}).

\subsection{\emph{z}-polynomials}\label{z-poly-subsection}

Let $S$ and $T$ be multisets of integers.  Define a Laurent polynomial
$z(S,T)$ by the formula
  \begin{equation} \label{z-poly}
  z(S,T) = \frac{1}{\Up X_S} \sum_{j\in T} X_{S\cup j},
  \end{equation}
where as usual the symbol $\cup$ denotes multiset union.  
An example was given at the end of Example~\ref{ex:bipyramid}.

\begin{proposition} \label{z-recurrences}
Let $d>i$ be integers, and let $S,T$ be sets of integers greater than $p$.
Then
  \begin{equation} \label{del-recurrence}
  X_{d-i+1,p} + \Up^{d-i}z(S,T) ~=~ \Up^{d-i}z(S,\tilde T)
  \end{equation}
and
  \begin{equation} \label{link-recurrence}
  X_{d-i+1,p} + \frac{X_{d-i+1,p}}{X_{d-i+2,p}} \Up^{d-i+1}z(S,T) ~=~ \Up^{d-i}z(\tilde S,\tilde T)
  \end{equation}
where $\tilde S = S\cup\{p\}$ and $\tilde T = T\cup\{p\}$.
\end{proposition}

\begin{proof}
We will use the identity \eqref{X-ident} repeatedly in the calculations.
For \eqref{del-recurrence}, observe that
  \begin{align*}
  X_{d-i+1,p} + \Up^{d-i} z(S,T)
  &= X_{d-i+1,p} + \frac{1}{\Up^{d-i+1} X_S} \sum_{j\in T} \Up^{d-i} X_{S\cup j}\\
  &= \frac{1}{\Up^{d-i+1} X_S} \left(
    X_{d-i+1,p} \cdot \Up^{d-i+1} X_S + \sum_{j\in T} \Up^{d-i} X_{S\cup j} \right)\\
  &= \frac{1}{\Up^{d-i+1} X_S} \left(
    \Up^{d-i} X_{\tilde S} + \sum_{j\in T} \Up^{d-i} X_{S\cup j} \right)\\
  &= \frac{1}{\Up^{d-i+1} X_S} \sum_{j\in \tilde T} \Up^{d-i} X_{S\cup j}\\
  &= \Up^{d-i} z(S,\tilde T),
  \end{align*}
and for \eqref{link-recurrence}, observe that
  \begin{align*}
  X_{d-i+1,p} + \frac{X_{d-i+1,p}}{X_{d-i+2,p}} & \Up^{d-i+1}z(S,T)
  = X_{d-i+1,p} + \frac{X_{d-i+1,p}}{X_{d-i+2,p}\cdot\Up^{d-i+2} X_S} \sum_{j\in T} \Up^{d-i+1} X_{S\cup j}\\
  &= X_{d-i+1,p} + \frac{X_{d-i+1,p}}{\Up^{d-i+1} X_{\tilde S}} \sum_{j\in T} \Up^{d-i+1} X_{S\cup j}\\
  &= \frac{1}{\Up^{d-i+1} X_{\tilde S}} \left(
    X_{d-i+1,p}\cdot\Up^{d-i+1} X_{\tilde S} + \sum_{j\in T} X_{d-i+1,p}\cdot\Up^{d-i+1} X_{S\cup j} \right)\\
  &= \frac{1}{\Up^{d-i+1} X_{\tilde S}} \left(
    \Up^{d-i} X_{\tilde S\cup p} + \sum_{j\in T} \Up^{d-i} X_{\tilde S\cup j} \right)\\
  &= \frac{1}{\Up^{d-i+1} X_{\tilde S}} \sum_{j\in \tilde T} \Up^{d-i} X_{\tilde S\cup j}\\
  &= \Up^{d-i} z(\tilde S,\tilde T).
  \end{align*}
\end{proof}

\begin{theorem} \label{shifted-recurrence}
Let $\Sigma^d$ be a shifted simplicial complex.  Then every nonzero eigenvalue of $\Sud_{i-1}(\Sigma)$ has the form of a $z$-polynomial.
Moreover, the spectrum~$\Sud_{i-1}(\Sigma)$ is determined recursively as follows.

If $\Sigma$ has no vertices, then $\Sud_{i-1}(\Sigma)$ has no nonzero elements.

Otherwise,
  \begin{equation} \label{shiftedEVals:nearcone}
  \begin{aligned}
  \Sud_{i-1}(\Sigma) &\zeq
  \left\{\Up^{d-i}z(S,\tilde T) \st  0 \neq z(S,T)\in \Sud_{i-1}(\Delta)\right\}\\
  &\cup~ \left\{\Up^{d-i}z(\tilde S, \tilde T) \st 0 \neq z(S,T)\in \Sud_{i-2}(\Lambda)\right\}
  ~\cup~ \left\{\Up^{d-i}z(\0,\tilde \0) \right\}^{\betti_{i-1}(\Delta)}
  \end{aligned}
  \end{equation}
where $p$ is the initial vertex of $\Sigma_{[i]}$; $\tilde S = S \cup p$; $\tilde T = T \cup p$; $\Delta = \del_p \Sigma_{[i]}$; and $\Lambda = \link_p \Sigma_{[i]}$.
\end{theorem}

\begin{proof}
The proof is by induction on the number of vertices of $\Sigma$.  When
$\Sigma$ has no vertices, $\Sigma$ is either the empty complex with no
faces, or the trivial complex whose only face is the empty face.  In
either case $\Sud_{i-1}(\Sigma)$ has no nonzero elements.

We now assume that $\Sigma$ has at least one vertex.  Then
\begin{align*}
\Sud_{i-1}(\Sigma_{[i]}) & \zeq \left\{ X_{1,p}+\lambda \st
     \lambda\in\Sud_{i-1}(\Delta),\ \lambda\neq 0 \right\}\\
    &\qquad\cup \left\{ X_{1,p}+\frac{X_{1,p}}{X_{2,p}} \Up\mu \st
     \mu\in\Sud_{i-2}(\Lambda),\ \mu\neq 0 \right\}\\
    &\qquad\cup \left\{ X_{1,p} + 0\right\}^{\betti_{i-1}(\Delta)}\\
&= \left\{ X_{1,p}+z(S,T) \st
     z(S,T)\in\Sud_{i-1}(\Delta),\  z(S,T) \neq 0 \right\}\\
    &\qquad\cup \left\{ X_{1,p}+\frac{X_{1,p}}{X_{2,p}} z(S,T) \st
     z(S,T)\in\Sud_{i-2}(\Lambda),\ z(S,T)\neq 0 \right\}\\
    &\qquad\cup \left\{ X_{1,p} + z(\0,\0) \right\}^{\betti_{i-1}(\Delta)}\\
 &=  \left\{ z(S,\tilde T) \st
     0 \neq z(S,T) \in \Sud_{i-1}(\Delta) \right\}\\
    &\qquad\cup \left\{ z(\tilde S, \tilde T) \st
    0 \neq z(S,T) \in \Sud_{i-2}(\Lambda) \right\}\\
    &\qquad\cup \left\{z(\0,\tilde \0) \right\}^{\betti_{i-1}(\Delta)}.
\end{align*}
The $\zeq$-equivalence above is by Proposition~\ref{near-cone-recurrence}.  
The following equality is justified by the identity $z(\0,\0)=0$ and induction
on the number of vertices, since 
$\Delta$ and $\Lambda$ each have one fewer vertex than $\Sigma$.  Note that $\lambda$ and $\mu$ are each replaced by $z(S,T)$, with no raising operator, because 
$\dim \Delta \leq i$ and $\dim \Lambda = i-1$.  The final equality comes from
Proposition~\ref{z-recurrences}.
The result now follows from Corollary~\ref{pure-Sud}.
\end{proof}

\subsection{Critical pairs}
Throughout this section, let $\family$ be a $k$-family of sets of integers,
and let $p$ be the smallest integer occurring in any element of $\family$.

\begin{definition}
A \emph{critical pair} for $\family$ is an ordered pair $(A,B)$,
where $A=\{a_1<a_2<\cdots<a_{k}\}$ and
$B=\{b_1<b_2<\cdots<b_{k}\}$ are sets of integers such that
$A\in\family$, $B\not\in\family$, and $B$ covers $A$ in
componentwise order.  That is, $b_i=a_i+1$ for exactly one $i$,
and $b_j=a_j$ for all $j\neq i$.  
(Note that $b_i$ need not be in the vertex set of $\family$.)
The \emph{signature} of $(A,B)$ is the set of vertices
  $\sg(A,B)=\{a_1,\ldots,a_{i-1},a_i\}$.
The \emph{long signature} is the ordered pair of sets
  $\lsg(A,B)=(S,T)$,
where $S=\{a_1,\ldots,a_{i-1}\}$ and $T=\{j\st p\leq j\leq a_i\}$.
The multisets of signatures and long signatures of critical pairs of $\family$ are denoted
$\setsg(\family)$ and $\setlsg(\family)$ respectively.
\end{definition}

As described in the Introduction and Example~\ref{ex:bipyramid}, critical pairs are
especially significant for shifted simplicial complexes.
We will soon see that the critical pairs of a shifted complex are in bijection
with the eigenvalues of its algebraic finely weighted Laplacian.

\begin{definition}
The \emph{degree} of vertex $v$ in the family $\family$ is
$\deg_{\family}(v) = \abs{\{F \in \family \st v \in F\}}$.
\end{definition}

\begin{proposition}\label{degree-signature}
Let $\family$ be a shifted family.  Then $\deg_{\family}(v) - \deg_{\family}(v+1)$ 
counts the number of signatures of $\family$ 
whose greatest element is $v$.
\end{proposition}

\begin{proof}
Let $S=\{F\in\family\st v\in F\}$ and $T=\{F\in\family\st v+1\in F\}$.
Partition $S = S_1\dju S_2\dju S_3$ and $T=T_1\dju T_2\dju T_3$ as follows:
  $$
  \begin{array}{l}
  S_1 = \{F \in \family\st v,v+1\in F\},\\
  S_2 = \{F \in \family\st v\in F,\ v+1\not\in F,\ F\sm\{v\}\dju\{v+1\}\in\family\},\\
  S_3 = \{F \in \family\st v\in F,\ v+1\not\in F,\ F\sm\{v\}\dju\{v+1\}\not\in\family\},\\\\
  T_1 = \{F \in \family\st v,v+1\in F\},\\
  T_2 = \{F \in \family\st v\not\in F,\ v+1\in F,\ F\sm\{v+1\}\dju\{v\}\in \family\},\\
  T_3 = \{F \in \family\st v\not\in F,\ v+1\in F,\ F\sm\{v+1\}\dju\{v\}\not\in \family\}.
  \end{array}
  $$
Then $S_1 = T_1$, and there is an obvious bijection between 
$S_2$ and $T_2$.  Since $\family$ is shifted, 
$T_3=\0$, so 
$\deg_{\family}(v) - \deg_{\family}(v+1) = \abs{S} - \abs{T} = \abs{S_{3}}$.  

Finally, if $F\in S_3$, then $(F,G)$ is a critical pair,
where $G = F \sm \{v\} \dju \{v+1\}$.  For each such critical pair, the
greatest element of the signature is $v$.  Conversely, if $(A,B)$
is a critical pair whose signature's greatest element is $v$, then
$A \in S_3$.
\end{proof}

\begin{corollary}\label{degree-beta}
Let $\Sigma^d$ be a pure shifted simplicial complex with
initial vertex $p$.  Then 
$\deg_{\Sigma_d}(p) - \deg_{\Sigma_d} (p+1) = \betti_{d-1}(\del_p \Sigma)$.
\end{corollary}

\begin{proof}
Just as in the proof of Proposition~\ref{degree-signature} above, 
  $$
  \deg_{\Sigma_{d}}(p) - \deg_{\Sigma_{d}}(p+1) = \abs{S_3},
  $$
where  
  $$
  S_3 = \{F\in\Sigma_d \st p\in F,\ p+1\not\in F,\ F\sm\{p\}\dju\{p+1\}\not\in\Sigma_d \}.
  $$
There is a bijection between $S_3$ and the set
  $$
  S'_3 = \{G\in(\del_p\Sigma)_{d-1} \st p+1 \not\in G,\ G\dju\{p+1\}\not\in\del_p\Sigma\}
  $$
given by $G=F\sm\{p\}$.  Then, by equation~\eqref{beta-shifted},
  $$
  \betti_{d-1}(\del_p\Sigma) = \abs{S'_3}.
  $$
Combining the four displayed equations yields the desired result.
\end{proof}

\begin{proposition}\label{critpairs-recurrence}
Let $\Sigma$ be a shifted complex with initial vertex $p$.  Let $\Delta = \del_p \Sigma$ and 
$\Lambda = \link_p \Sigma$.  Then
  $$
  \sg(\Sigma_i) = \sg(\Delta_i) \cup \{p \dju F\st F \in \sg(\Lambda_{i-1})\}
  \cup \{p\}^{\deg_{\Sigma_i}(p) - \deg_{\Sigma_i}(p+1)},
  $$
where $\cup$ denotes multiset union.
\end{proposition}

\begin{proof}
The multiplicity of the signature $\{p\}$ follows from
Proposition~\ref{degree-signature}.
  
Suppose that $A\in\Delta_i$, $B\not\in\Delta_i$, and $B$ covers 
$A$ in the componentwise partial order.  Then $A\in\Sigma_i$ and
$p\not\in A$.  Since $p\not\in A$ and $A\cpless B$, we conclude that $p\not\in B$,
and so $B\not\in\Sigma_i$.  Hence we have a one-to-one map of multisets
  $$\sg(\Delta_i) \rightarrow \sg(\Sigma_i).$$

On the other hand, suppose that $A\in\Lambda_{i-1}$, $B\not\in\Lambda_{i-1}$, and
$B$ covers $A$ in the componentwise partial order.  Then $p\not\in A$
and $\tilde A=A\dju\{p\}\in\Sigma_i$.  The critical pair $(A,B)$ of $\Lambda_{i-1}$
gives rise to the critical pair $(\tilde A,\tilde B)$ of $\Sigma_i$, and it is 
clear that $\sg(\tilde A,\tilde B) = \sg(A,B)\dju\{p\}$.  Furthermore,
$\tilde B\not\in\Sigma_i$ because $B\not\in\Lambda_{i-1}$.  Hence we have a
one-to-one map of multisets
  $$\{p \dju F\st F \in \sg(\Lambda_{i-1})\} \rightarrow \sg(\Sigma_i).$$

Now we must show, conversely, that every signature $F \neq \{p\}$ of $\Sigma_i$
arises in one of these two ways.
Let $F$ be such a signature of $\Sigma_i$, with critical pair $(A,B)$.
So $A\in\Sigma_i$;
$B \not\in \Sigma_i$ (so $B\not\in\Delta_i$); and $B$ covers $A$ in the componentwise partial order.

First, if $p\not\in F$, then $p \not\in A$ and so $A \in \Delta_i$.
Thus $(A,B)$ is a critical pair for $\Delta_i$.

Second, if $p\in F$, then $F = F'\dju\{p\}$ for some $F'\neq\0$.
In this case, we have $p \in A$ and $p \in B$ for 
the critical pair $(A,B)$ whose signature is $F$.
Indeed, $p \in F$ directly implies that $p \in A$.
If $p\not\in B$, then $F = \sg(A,B) = \{p\}$.  Accordingly,
let $A'=A\sm\{p\}$ and $B'=B\sm\{p\}$.  Then $A'\in\Lambda_{i-1}$
and $B\not\in\Sigma_i$, so $B'\not\in\Lambda_{i-1}$, and $(A,B)$
is a critical pair for $\Lambda_{i-1}$.
\end{proof}

\begin{corollary}\label{allsg-shifted}
Let $\Sigma^d$ be a shifted simplicial complex,
and let $i\leq d$.  Then the multiset $\setlsg(\Sigma_i)$ is determined 
by the following recurrence.

If $\Sigma$ has no vertices, then $\setlsg(\Sigma_i) = \0$.

Otherwise,
\begin{equation} \label{multiset-recurrence}
\begin{aligned}
\setlsg(\Sigma_i) &= \{(S, \tilde T)\st (S,T) \in \setlsg(\Delta_i)\}\\
  &\qquad\cup \{(\tilde S, \tilde T)\st (S,T) \in \setlsg(\Lambda_{i-1})\}\\
  &\qquad\cup \{(\0,\tilde \0)\}^{\betti_{i-1}(\Delta)}
  \end{aligned}
\end{equation}
where $p$ is the initial vertex of $\Sigma_{[i]}$; $\tilde S = S \cup p$; $\tilde T = T \cup p$;
$\Delta = \del_p \Sigma_{[i]}$; and $\Lambda = \link_p \Sigma_{[i]}$.
\end{corollary}

\begin{proof}
If $\sg(A,B) = \{p\}$, then $\lsg(A,B) = (\0,\{p\}) = (\0,\tilde\0)$.
So the third term in \eqref{multiset-recurrence} arises from applying
Corollary~\ref{degree-beta} to the pure $i$-dimensional shifted
complex $\Sigma_{[i]}$. (Indeed, $\Sigma_{[i]}$ is shifted when $\Sigma$ is,
or even just when $\Sigma_i$ is.)

For the first two terms in \eqref{multiset-recurrence}, the only hard part is to note that if the first vertex of $\Sigma_i$ is $p$, then $p+1$ is the first vertex of $\Delta_i$ and $\Lambda_{i-1}$ (unless $\Delta_i = \0$, but in that case we don't have to worry about the first set).  This accounts for the $\tilde T$'s.

Even though Proposition~\ref{critpairs-recurrence} above defines $\Delta$ and $\Lambda$ somewhat differently than here, it is not a problem because
$(\del_p \Sigma)_i = (\del_p \Sigma_{[i]})_i$
and
$(\link_p \Sigma)_{i-1} = (\link_p \Sigma_{[i]})_{i-1}$.
Then, since $\Delta$ and $\Lambda$ only appear as
$\Delta_i$ and $\Lambda_{i-1}$, it doesn't matter whether we
set them to be the deletion and link of $\Sigma$ or of $\Sigma_{[i]}$.
\end{proof}

\subsection{Theorem~\ref{zpoly-theorem} and its consequences}

We can finally prove Theorem~\ref{zpoly-theorem}, which characterizes the
Laplacian spectra $\Sud_{i-1}(\Sigma)$ in terms of $z$-polynomials of
critical pairs.  In the notation we have developed, the theorem can be
restated as follows:

\setcounter{savesection}{\value{section}}
\setcounter{section}{\value{intro-section-counter}}
\setcounter{savetheorem}{\value{theorem}}
\setcounter{theorem}{\value{zpoly-counter}}
\begin{theorem}
Let $\Sigma^d$ be a shifted simplicial complex.  Then, for each $0 \leq i \leq d$,
\begin{equation} \label{the-big-one}
\Sud_{i-1}(\Sigma) \zeq\: \Up^{d-i}\{z(S,T)\st (S,T) \in \setlsg(\Sigma_{i})\}.
\end{equation}
\end{theorem}
\setcounter{section}{\value{savesection}}
\setcounter{theorem}{\value{savetheorem}}

\begin{proof}
Simply note that the recursions in Theorem~\ref{shifted-recurrence} and Corollary~\ref{allsg-shifted}
are identical.
\end{proof}

One corollary to Theorem~\ref{zpoly-theorem} is that you can ``hear the shape'' of a shifted complex (Corollary~\ref{fine-hearing}), but only if your ears are fine enough (Remark~\ref{coarse-not-hearing}).

\begin{corollary}\label{fine-hearing}
A shifted complex $\Sigma^d$ is completely determined by its spectra $\{\Sud_{i-1}(\Sigma)\}_{i=0}^d$.
\end{corollary}

\begin{proof}
By Theorem~\ref{zpoly-theorem}, the spectra $\{\Sud_{i-1}(\Sigma)\}_{i=0}^d$ 
determine the $z$-polynomials $z(S,T)$,
and it is easy to see that the long signature $(S,T)$ can be recovered from $z(S,T)$.
Furthermore the (short) signature is even more easily recovered from the long signature.
When $F=\{v_1<\ldots<v_i\}$ is a face of $\Sigma$ that is $\cpleq$-maximal, then 
$(F, F \sm \{v_i\} \dju \{v_{i+1}\})$ is a critical pair, with (short) signature $F$.
Thus, among all the (short) signatures of $\Sigma$, we will find all $\cpleq$-maximal faces.
Furthermore, every (short) signature is a face of $\Sigma$, by definition.  Thus, the union
$\cup_{F \in \sg(\Sigma)} \{G\colon G \cpleq F\}$
of all $\cpleq$-order ideals will yield all the non-empty faces of $\Sigma$.
\end{proof}

In fact, if $\Sigma^d$ is a \emph{pure} shifted complex, then
it is determined uniquely by its top Laplacian spectrum $\Sud_{d-1}(\Sigma)$.

Specializing (or ``coarsening'') the algebraic fine weighting, we obtain as another 
corollary to Theorem~\ref{zpoly-theorem} Duval and Reiner's description of the unweighted 
Laplacian eigenvalues of a shifted complex \cite[Thm.~1.1]{DR}, as we now explain.

The \emph{coarse weighting} is obtained from the algebraic fine weighting by
omitting all first subscripts, i.e., replacing $x_{i,j}=x_j$ and $X_{i,j}=X_j$.  (Thus the
monomial corresponding to a facet or set of facets records the degree of
each vertex, but forgets the information about the order of vertices in facets.)
Note that if $T=[1,t]$, then every $z$-polynomial $z(S,T)$ specializes to 
the linear form~$E_t=X_1 + \cdots + X_t$ in the coarse weighting.

For a partition~$\lambda$ (a weakly decreasing list of positive integers), 
let $E_{\lambda}$ be the multiset in which each part~$i$ of 
$\lambda$ is replaced by~$E_i$.  
Recall that the \emph{conjugate} of
$\lambda$ is the partition $\lambda'$ in which each part~$t$ occurs with
multiplicity $\lambda_t - \lambda_{t+1}$.

\begin{corollary}\label{Swud-shifted}
Let $\Sigma^d$ be a shifted simplicial complex on vertices~$[n]$.  Then
  $$
  \Swud_{d-1}(\Sigma) \zeq E_{(\deg_{\Sigma_d})'}
  $$
where the left-hand side denotes the multiset of coarsely weighted Laplacian eigenvalues, and the right-hand side is the conjugate of the partition
$(\deg_{\Sigma_d}(1),\dots,\deg_{\Sigma_d}(n))$.
\end{corollary}

\begin{proof}
Theorem~\ref{zpoly-theorem} and the preceding discussion imply that
$\Swud_{d-1}(\Sigma)$ is $\zeq$-equivalent to the multiset in which
$E_t$ occurs with multiplicity
equal to the number of critical pairs
$(A,B)$ of $\Sigma_d$ such that $t = \max\sg(A,B)$.
By Proposition~\ref{degree-signature}, that multiplicity is
$\deg_{\Sigma_d}(t) - \deg_{\Sigma_d}(t+1)$.  The result now follows from
the definition of conjugate partition.
\end{proof}

Passing to the unweighted setting by setting $x_i=1$ for all~$i$ recovers
the theorem of Duval and Reiner~\cite[Thm.~1.1]{DR}, which states that
the Laplacian eigenvalues of a shifted complex
$\Sigma$ are given by the conjugate of the partition $\deg_{\Sigma_d}$.

\begin{remark}\label{coarse-not-hearing}
Duval and Reiner also showed \cite[Example~10.2]{DR} that there are two non-isomorphic
2-dimensional shifted complexes with the same degree sequence.  Corollary~\ref{Swud-shifted}
then shows that, in contrast to Corollary~\ref{fine-hearing}, the coarsely-weighted
eigenvalues are not enough to determine a shifted complex.
\end{remark}
\subsection{An example: the equatorial bipyramid}
\label{sec:example}

To illustrate Theorems~\ref{shifted-recurrence} and~\ref{zpoly-theorem}, we calculate the top-dimensional 
up-down Laplacian spectrum of the bipyramid $B=\shgen{235}$
(see Examples~\ref{ex:bipyramid},~\ref{ex:shifted235} and~\ref{ex:weightedsst235}).
In our recursive calculation, we will encounter the subcomplexes $B_2,\dots,B_7$
of $B=B_1$ shown in the following figure, and $B_8$, the simplicial complex whose
only face is the empty face.

\begin{center}
\resizebox{5.0in}{1.3in}{\includegraphics{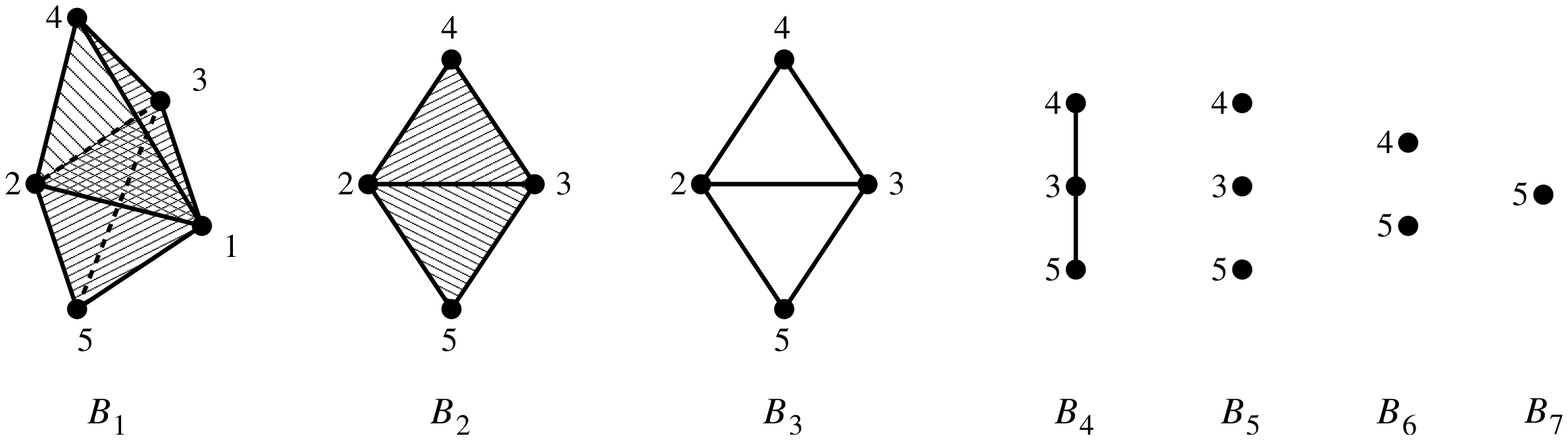}}
\end{center}

Observe that
\begin{itemize}
\item $B_1$ is a near-cone with apex 1, $\del_1 B_1=B_2$, and $\link_1 B_1=B_3$;
\item $B_2 = 2*B_4$;
\item $B_3$ is a near-cone with apex 2, $\del_2 B_3=B_4$, and $\link_2 B_3=B_5$;
\item $B_4 = 3*B_6$;
\item $B_5$ is a near-cone with apex 3, $\del_3 B_5 = B_6$, and $\link_3 B_5 = B_8$;
\item $B_6$ is a near-cone with apex 4, $\del_4 B_6 = B_7$, and $\link_4 B_6 = B_8$; and
\item $B_7 = 5*B_8$.
\end{itemize}

The critical pairs, signatures, and nonzero top-dimensional Laplacian eigenvalues of the bipyramid $B=B_1$ and its subcomplexes $B_2,\dots,B_7$ are listed in the following table.  This information can be obtained either recursively (using Theorem~\ref{shifted-recurrence} repeatedly) or bijectively (from Theorem~\ref{zpoly-theorem}).

$$
  \begin{array}{cccccc}
    \text{Subcomplex} & \text{Dimension} & \text{Vertices} & \text{Critical pairs} & \text{Signatures} & \text{Eigenvalues} \\ 
    \hline
    B_7 & 0 & \{5\} &
      (5,6) & 5 & z(\0,5)\\  \hline
    B_6 & 0 & \{4,5\} &
      (5,6) & 5 & z(\0,45)\\ \hline
    B_5 & 0 & \{3,4,5\} &
      (5,6) & 5 & z(\0,345)\\ \hline
    B_4 & 1 & \{3,4,5\} &
      (35,45) & 3  & z(\0,3)\\ &&&
      (35,36) & 35 & z(3,345)\\ \hline
    B_3 & 1 & \{2,3,4,5\} &
      (25,26) & 25 & z(2,2345) \\ &&&
      (35,36) & 35 & z(3,2345) \\ &&&
      (35,45) & 3  & z(\0,23) \\ \hline
    B_2 & 2 & \{2,3,4,5\} &
      (235,236) & 235 & z(23,2345)\\ &&&
      (235,245) & 23  & z(2,23)\\ \hline
    B_1 & 2 & \{1,2,3,4,5\} &
      (125,126) & 125 & z(12,12345) \\ &&&
      (135,136) & 135 & z(13,12345) \\ &&&
      (135,145) & 13  & z(1,123) \\ &&&
      (235,236) & 235 & z(23,12345) \\ &&&
      (235,245) & 23 & z(2,123) \\ 
    \hline
  \end{array}
$$

We also see from the above table that the coarsely-weighted eigenvalues $\Swud_1(B_1)$ are 
$\zeq$-equivalent to $E_5, E_5, E_5, E_3, E_3$ (each $E_5$ coming from a $z(S,12345)$ 
in $B_1$, and each $E_3$ coming from a $z(S,123)$ in $B_1$), corresponding to the transpose
of the degreee sequence of the facets, $55533$.  (In this case, both the degree sequence
and its transpose are $55533$.)

\section{Enumerating spanning trees of shifted complexes}
\label{count-shifted-section}

We now translate Theorem~\ref{zpoly-theorem} from the algebraic to the
combinatorial fine weighting, in order to obtain a factorization of the
finely weighted spanning tree enumerator~$\hat\tau(\Sigma)$ of a
shifted complex~$\Sigma$.

Recall that the long signature~$\setlsg(\family)$ of a family~$\family$
is the multiset of long signatures of its critical pairs.

\setcounter{savesection}{\value{section}}
\setcounter{section}{\value{intro-section-counter}}
\setcounter{savetheorem}{\value{theorem}}
\setcounter{theorem}{\value{shifted-counter}}
\begin{theorem}
Let $\Sigma^d$ be a shifted complex with initial vertex
$p$.  Then:
\begin{align}
\hat{\tau}_d(\Sigma) 
    &= \left( \prod_{ F \in \Lambda_{d-1}} X_{\tilde F} \right)
        \left(  \prod_{(S,T) \in \setlsg(\Delta_d)} \frac{z(S,\tilde T)}{X_{1,p}}  \right)
    \label{h-shifted-z}\\
    &= \left( \prod_{ F \in \Lambda_{d-1}} X_{\tilde F} \right)
        \left( \prod_{(S,T) \in \setlsg(\Delta_d)} 
        \frac{\sum_{j \in \tilde T} X_{S \cup j}}{X_{\tilde S}} \right)
    \label{h-shifted}
\end{align}
where $\tilde F = F \cup p,\ \Delta = \del_p \Sigma$, and $\Lambda = \link_p \Sigma$.
\end{theorem}
\setcounter{section}{\value{savesection}}
\setcounter{theorem}{\value{savetheorem}}

\begin{proof}
Since the complex $\Sigma$ is APC, its spanning trees are precisely those
of its pure $d$-skeleton $\Sigma_{[d]}$.  
Similarly, passing from $\Sigma$ to $\Sigma_{[d]}$ does not affect $\Delta_d$ 
or $\Lambda_{d-1}$.  Therefore, we may assume without loss of generality
that $\Sigma$ is pure.

By Lemma~\ref{reduce-to-evals} and Theorem~\ref{zpoly-theorem}, we have
$$
\hat{\tau}_d(\Sigma) = 
         \left(\prod_{ F\in\link_p\Sigma} \Up X_F\right)
         \left(\prod_{(S,T) \in \setlsg(\Delta_d)} (X_{1,p} + z(S,T))\right)X_{1,p}^m
$$
where $m$ is the number of zero eigenvalues of $\LLud_{\Delta,d-1}$.  Since $\LLud_{\Delta,d-1}$ acts on $\Delta_{d-1}$, but has $\setlsg(\Delta_d)$ nonzero eigenvalues (including multiplicity), $m = \abs{\Delta_{d-1}} - \abs{\setlsg(\Delta_d)} = \abs{\Lambda_{d-1}} - \abs{\setlsg(\Delta_{d-1})}$, since $\Delta_{d-1} = \Lambda_{d-1}$ by equation~\eqref{Delta-Lambda}.  Thus, 
  \begin{align*}
  \hat \tau_d(\Sigma) &= X_{1,p}^{\abs{\Lambda_{d-1}}}
    \left(\prod_{ F \in \Lambda_{d-1}}(\Up X_{ F}) \right) X_{1,p}^{-\abs{\setlsg(\Delta_d)}}
    \left(\prod_{(S,T) \in \setlsg(\Delta_d)} X_{1,p} + z(S,T)\right) \\
  &= \prod_{ F \in \Lambda_{d-1}} (X_{1,p} \Up X_{ F})
    \prod_{(S,T) \in \setlsg(\Delta_d)} \frac{X_{1,p} + z(S,T)}{X_{1,p}}.
\end{align*}
Equation~\eqref{h-shifted-z} now follows from
equations~\eqref{X-ident} and~\eqref{del-recurrence}.  
Equation~\eqref{h-shifted} then follows from
the definition of $z$-polynomial~\eqref{z-poly} and from~\eqref{X-ident} again, because
$$
 \frac{z(S,\tilde T)}{X_{1,p}}
 =  \frac{1}{X_{1,p}}\frac{\sum_{j \in \tilde T} X_{S \cup j}}{\Up X_{S}}
 = \frac{\sum_{j \in \tilde T} X_{S \cup j}}{X_{\tilde S}}.
$$
\end{proof}

\begin{example}\label{ex:fine-trees-bipyramid}
We return to our running example, the equatorial bipyramid $B$.  Here $d=2$ and $p=1$, and in the notation of Section~\ref{sec:example}, we have $\Delta=B_2$ and  $\Lambda=B_3$.
Moreover, $\setlsg(\Delta_d) = \{(2,23),(23,2345)\}$.
Hence equation~\eqref{h-shifted} yields
$$
\hat\tau(B) = X_{123}X_{124}X_{134}X_{125}X_{135} \left(\frac{X_{12} + X_{22} + X_{23}}{X_{12}}\right)
\left(\frac{X_{123} + X_{223} + X_{233} + X_{234} + X_{235}}{X_{123}}\right).
$$
Note that this is a genuine polynomial (not just a Laurent polynomial) in the indeterminates $X_{i,j}$.
\end{example}

\begin{corollary}\label{coarse-shifted}
Let $\Sigma^d$ be a shifted complex with initial vertex~$1$.
Let $\Delta = \del_1 \Sigma$, $\tilde\Delta = 1 * \Delta$, 
$\Lambda = \link_1 \Sigma$, and $\tilde\Lambda = 1 * \Lambda$.
Then, in the coarse weighting, 
$$
\hat \tau_d(\Sigma) = X^{\deg \tilde\Lambda_d} \prod_i 
          ( E_i/X_1)^{(\deg \tilde\Delta_{d+1})'_i},
$$
where $E_i = X_1 + \cdots + X_i$ and, for a partition $\lambda$, we set $X^\lambda := \prod_i X_i^{\lambda_i}$.
\end{corollary}

\begin{proof}
Upon coarsening the weighting, the first product in~\eqref{h-shifted} becomes
$X^{\deg \tilde\Lambda_d}$ and the second product becomes
  $$
  \prod_{(S,T)\in \setlsg(\Delta_d)} \frac{\sum_{j \in \tilde T}X_j}{X_1}
  = \prod_{(S,T)\in \setlsg(\Delta_d)} E_{\abs{T}+1}/X_1.
  $$
We now claim that
  $$
  \prod_{(S,T)\in \setlsg(\Delta_d)} E_{\abs{T}+1}/X_1
  = \prod_{(S,T)\in \setlsg(\tilde\Delta_{d+1})} E_{\abs{T}}/X_1.
  $$
Indeed, by Proposition~\ref{critpairs-recurrence}, 
  $$
  \setsg(\tilde\Delta_{d+1}) = \{1 \dju F \colon F \in \setsg(\Delta_d)\} \cup \{1\}^m
  $$
for some $m$, since $\Delta = \link_1(\tilde\Delta)$ and 
$\dim(\del_1(\tilde\Delta)) = \dim \Delta < d+1$.  It now follows that 
  $$
  \prod_{(S,T)\in \setlsg(\tilde\Delta_{d+1})} E_{\abs{T}}/X_1
  = \prod_{(S,T)\in \setlsg(\Delta_d)} (E_{\abs{T}+1}/X_1) (X_1 / X_1)^m,
  $$
implying the claim.
Finally, by Proposition~\ref{degree-signature} and the definition of conjugate partition,
we obtain
  $$
  \prod_{(S,T)\in \setlsg(\tilde\Delta_{d+1})} E_{\abs{T}}/X_1
  = \prod_t (E_t / X_1)^{\deg_{\tilde\Delta_{d+1}}(t) - \deg_{\tilde\Delta_{d+1}}(t+1)}
  = \prod_i ( E_i/ X_1)^{(\deg \tilde\Delta_{d+1})'_i}.
  $$
\end{proof}

\begin{example}\label{ex:coarse-trees-bipyramid}
Once again, let $B$ be the equatorial bipyramid.  Here $\tilde\Lambda_d=\{123, 124, 125, 134, 135\}$.
Its degree sequence is~53322.  Thus the monomial factor in Corollary~\ref{coarse-shifted} is
  $$
  X_{123}X_{124}X_{125}X_{134}X_{135} = X_1^5 X_2^3 X_3^3 X_4^2 X_5^2.
  $$
Meanwhile, $\tilde\Delta_{d+1}=\{1234, 1235\}$.  Its degree sequence is~22211, with conjugate~53.  The product factor in Corollary~\ref{coarse-shifted} is therefore
  $$
  (E_5/X_1)(E_3/X_1) = (X_1 + X_2 + X_3 + X_4 + X_5)(X_1 + X_2 + X_3)/X_1^2.
  $$
Putting these terms together yields
  $$
  \hat \tau_d(\Sigma) = X_1^3 X_2^3 X_3^3 X_4^2 X_5^2
  (X_1 + X_2 + X_3 + X_4 + X_5) (X_1 + X_2 + X_3)
  $$
which matches Example~\ref{ex:weightedsst235}.
\end{example}

\section{Corollaries} \label{corollary-section}

We conclude by showing how several known tree enumerators---for skeletons
of simplices, threshold graphs, and Ferrers graphs---can be recovered from our results.

\subsection{Skeletons of simplices}
Let $\Sigma$ be the $d$-skeleton of the simplex on vertices~$[n]$,
so that the set of facets of $\Sigma$ is $\binom{[n]}{d+1}$,
the set of all subsets of $[n]$ of cardinality~$d+1$.  Note that
$\Sigma$ is generated as a shifted complex by the single facet $[n-d,n]$.
The critical pairs of $\Sigma$ are
  $$\left\{ (A\cup\{n\},\; A\cup\{n+1\}) \st A\in\binom{[n-1]}{d} \right\}$$
and the corresponding long signatures are
  $$\setlsg(\Sigma_d) = \left\{ (A,\,[n]) \st A\in\binom{[n-1]}{d} \right\}.$$
Setting $\Lambda = \link_1\Sigma$ and $\Delta=\del_1\Sigma$, we have
  \begin{align*}
  \Lambda_{d-1} &= \binom{[2,n]}{d};\\
  \Delta_d &= \binom{[2,n]}{d+1};\\
  \setlsg(\Delta_d) &= \left\{ (B,\,[2,n]) \st B\in\binom{[2,n-1]}{d} \right\}.
  \end{align*}
Applying equation~\eqref{h-shifted}, we obtain
  $$\hat \tau_d(\Sigma) = \left( \prod_{\substack{C\subseteq[2,n]\\|C|=d}} X_{\tilde C} \right)
    \left( \prod_{\substack{B\subseteq[2,n-1]\\\abs{B}=d}} \frac{\sum_{j=1}^n X_{B\cup j}}{X_{\tilde B}} \right).$$
The denominators in the second product cancel the factors $X_{\tilde C}$
in the first product with $n\not\in\tilde C$, leaving only those for which
$n\in\tilde C$.  Therefore,
  \begin{align*}
  \hat \tau_d(\Sigma) &= \left( \prod_{\substack{C\subseteq[2,n]\\n\in C\\\abs{C}=d}} X_{\tilde C} \right)
    \left( \prod_{\substack{B\subseteq[2,n-1]\\\abs{B}=d}} \sum_{j=1}^n X_{B\cup j} \right)
  &= \left( \prod_{\substack{C\subseteq[n]\\1,n\in C\\\abs{C}=d+1}} X_C \right)
    \left( \prod_{\substack{B\subseteq[2,n-1]\\\abs{B}=d}} z(B,[n]) \right).
  \end{align*}
Passing to the coarse weighting by setting $X_{i,j}=X_j$ for every $i,j$, we obtain
  \begin{align*}
  \hat \tau_d(\Sigma) &= \Big( (X_1X_n)^{\binom{n-2}{d-1}} (X_2\cdots X_{n-1})^{\binom{n-3}{d-2}} \Big)
    \Big( (X_2\cdots X_{n-1})^{\binom{n-3}{d-1}} (X_1+\cdots+X_n)^{\binom{n-2}{d}} \Big)\\
  &= (X_1\cdots X_n)^{\binom{n-2}{d-1}} (X_1+\cdots+X_n)^{\binom{n-2}{d}},
  \end{align*}
which is Theorem~$3'$ of Kalai's paper \cite{Kalai}.  Furthermore, setting $X_i=1$ for all~$i$ recovers
Kalai's generalization of Cayley's formula: $\tau_d(\Sigma)=n^{\binom{n-2}{d}}$.

\subsection{Threshold graphs}

A \emph{threshold graph} is a one-dimensional shifted complex~$\Sigma$.
For simplicity, we assume that the vertex set of~$\Sigma$ is~$[1,n]$.
We may also assume that~$\Sigma$ is connected, so that every vertex 
is adjacent to vertex~$1$.
Martin and Reiner \cite[Theorem 4]{MR2} found a factorization of the
combinatorially finely weighted spanning tree enumerator of~$\Sigma$, 
which may be stated in our notation\footnote{Note the distinction 
between the variable $X_{i,j}$, which corresponds to vertex $j$ as the $i$th smallest vertex in a face, 
and the quadratic monomial $X_{\{i,j\}}$, which corresponds to the edge $\{i,j\}$, and 
which equals $X_{1,i}X_{2,j}$ if $i \leq j$ but equals $X_{1,j}X_{2,i}$ if $j \leq i$.} as:
  \begin{equation} \label{JLM-VSR}
  \hat\tau(\Sigma) = X_{\{1,n\}} \prod_{v=2}^{n-1} \sum_{j=1}^{(\deg\Sigma)'_v} X_{\{v,j\}}.
  \end{equation}
A somewhat more general result was obtained independently by Remmel and
Williamson~\cite[Theorem 2.4]{RW}.  We will show how this formula can be
recovered from Theorem~\ref{shifted-theorem}.

The first product in equation~\eqref{h-shifted} is just~$X_{\{1,2\}}X_{\{1,3\}}\cdots X_{\{1,n\}}$.
For the second product, we must identify the critical pairs
of~$\Delta = \del_1 \Sigma$.  Note that~$\Delta$ is a threshold graph with vertices~$[2,n]$.

As is often the case with threshold graphs (and their degree sequences, which we will soon encounter), we need to sort vertices by their relation to the size of the \emph{Durfee square} of $\Sigma$, the largest square that fits in the Ferrers diagram of its degree sequence.  The side length~$m$ of the Durfee square is the largest number such that $\{m,m+1\}$ is an edge of $\Sigma$.  If $m=1$, then $\Sigma$ is a star graph, the Ferrers diagram of its degree sequence is a hook, and equations~\eqref{h-shifted} and~\eqref{JLM-VSR} both easily reduce to 
$$
\hat \tau_d(\Sigma) = \prod_{v=2}^n X_{\{1,v\}}.
$$
Therefore, we henceforth assume that $m \geq 2$.
Note that every edge has at least one endpoint~$\leq m$ (because $\{m+1,m+2\}\not\in\Sigma$,
and that edge is the unique $\cpleq$-minimal edge with both endpoints~$>m$).

For each vertex $v$ of $\Sigma$, let
  $$w(v) = \max \{u\colon \{u,v\} \in \Sigma\}.$$
Note that if $v\leq m$, then $\{v,m+1\}\cpleq\{m,m+1\}\in\Sigma$,
so $w(v)\geq m$.  On the other hand, if $v>m$, then $w(v)<v$.

\begin{lemma} \label{lemma:crit-pair-Delta}
The critical pairs of $\Delta = \del_1 \Sigma$ are as follows.

For each $v_1\in[2,m]$, there is a ``type~I'' critical pair $(\{v_1,v_2\},\{v_1,v_2+1\})$
where $v_2=w(v_1)$.

For each $v_2\in[m+2,w(2)]$, there is a ``type~II'' critical pair $(\{v_1,v_2\},\{v_1+1,v_2\})$,
where $v_1=w(v_2)$.

Furthermore, every critical pair is of one of these forms.
\end{lemma}

\begin{proof}
It is immediate from the definition of $w(v)$ that each such pair
is critical.  Suppose now that $(A,B)$ is a critical pair of $\Delta$,
with $A=\{v_1<v_2\}$.  Then either $B=\{v_1,v_2+1\}$ or $B=\{v_1+1,v_2\}$.

If $B=\{v_1,v_2+1\}$, then we have already observed that $A$ has at least one
endpoint in $[2,m]$.  In particular, $v_1\leq m$, and the pair $(A,B)$ is of type~I.

If $B=\{v_1+1,v_2\}$,
then $A\in\Delta$ and $B\not\in\Delta$, so by definition $v_1=w(v_2)$.  Moreover,
$m+2 \leq v_2$ (because $v_1+1\leq v_2-1$, so $\{v_2-1,v_2\}\cpmoreeq B\not\in\Delta$)
and $v_2 \leq w(2)$ (because $\{2,v_2\} \cpleq A$, so $\{2,v_2\}\in\Delta$).
Hence the pair $(A,B)$ is of type~II.
\end{proof}

If $(A,B)$ is a critical pair of type~I, then
$\sg(A,B)=\{v_1,v_2\}$ and $\lsg(A,B)=(\{v_1\},[2,v_2])$.
If $(A,B)$ is a critical pair of type~II, then
$\sg(A,B)=\{v_1\}$ and $\lsg(A,B)=(\0,[2,v_1])$.
Therefore, formula \eqref{h-shifted} yields
\begin{align}
 \hat \tau_d(\Sigma) 
  &=
    \prod_{v=2}^n X_{\{1,v\}} 
    \prod_{v_1=2}^m \frac{\sum_{j=1}^{w(v_1)}X_{\{v_1,j\}}}{X_{\{1,v_1\}}}
    \prod_{v_2=m+2}^{w(2)} \frac{\sum_{j=1}^{w(v_2)}X_{\{j\}}}{X_{\{1\}}} \notag\\
  &=
    \prod_{v=2}^n X_{\{1,v\}} 
    \prod_{v_1=2}^m \frac{\sum_{j=1}^{w(v_1)}X_{\{v_1,j\}}}{X_{\{1,v_1\}}}
    \prod_{v_2=m+2}^n \frac{\sum_{j=1}^{w(v_2)}X_{\{j\}}}{X_{\{1\}}} \notag\\
  &=
    X_{\{1,n\}}
    \left( \prod_{v_1=2}^m X_{\{1,v_1\}} \frac{\sum_{j=1}^{w(v_1)}X_{\{v_1,j\}}}{X_{\{1,v_1\}}} \right)
    \left( \prod_{v_2=m+2}^n X_{\{1,v_2 - 1\}}\frac{\sum_{j=1}^{w(v_2)}X_{\{j\}}}{X_{\{1\}}} \right).
  \label{a-threshold-expression}
\end{align}
The second equality follows because $w(v)=1$ whenever $v>w(2)$, and
the third equality comes from redistributing most of the first product among the other two.  Now, when $v_2 > m + 1$, we have
$$
     X_{\{1,v_2 - 1\}}\frac{\sum_{j=1}^{w(v_2)}X_{\{j\}}}{X_{\{1\}}}
  = \frac{X_{1,1} X_{2,v_2-1}}{X_{1,1}} \sum_{j=1}^{w(v_2)}X_{1,j}
  = \sum_{j=1}^{w(v_2)}X_{1,j}X_{2,v_2-1}
  = \sum_{j=1}^{w(v_2)}X_{\{j,v_2-1\}}
$$
since $j \leq w(v_2) \leq m < v_2 - 1$.  Thus we may rewrite \eqref{a-threshold-expression} as
\begin{equation} \label{nearly-there}
\hat \tau_d(\Sigma) 
    = \left( \prod_{v_1=2}^m \sum_{j=1}^{w(v_1)}X_{\{v_1,j\}} \right)
      \left( \prod_{v_2=m+2}^n \sum_{j=1}^{w(v_2)}X_{\{j, v_2 - 1\}} \right)
      X_{\{1,n\}}.
\end{equation}

If $v_1 \leq m$, then $w(v_1)>v_1$, so $w(v_1)$ has degree at least $v_1$, as does every
vertex less than $w(v_1)$.  On the other hand, $\{v_1,w(v_1)+1\}\not\in\Sigma$, so
vertex $w(v_1)+1$ has degree less than $v_1$, as does every vertex greater than $w(v_1)+1$.
Therefore, $(\deg\Sigma)'_{v_1} = w(v_1)$.

Similarly, if $v_2 > m+1$, then $(\{w(v_2),v_2)\},\{w(v_2)+1, v_2\})$ is a critical pair, so
$w(v_2)$ has degree at least $v_2-1$, as does every vertex less than $w(v_2)$.  On the other hand,
vertex $w(v_2)+1$ has degree less than $v_2-1$, as does every vertex greater than $w(v_2)+1$.
Therefore, $(\deg\Sigma)'_{v_2-1} = w(v_2)$.

Using these observations to rewrite \eqref{nearly-there} in terms of
the partition $(\deg\Sigma)'$ recovers the Martin-Reiner formula \eqref{JLM-VSR}.

\subsection{From threshold graphs to Ferrers graphs}

Let $\lambda=(\lambda_1\geq\cdots\geq\lambda_\ell)$ be a partition.  The corresponding
\emph{Ferrers graph} is the bipartite graph with vertices $x_1,\dots,x_{\lambda_1},y_1,\dots,y_\ell$
and edges $\{x_iy_j \st i\leq\lambda_j\}$.  That is, the vertices correspond to rows and columns
of the Ferrers diagram of $\lambda$, and the edges to squares appearing in the diagram.
Ehrenborg and van~Willigenburg \cite{EvW} considered Ferrers graphs and (among other results)
described how a certain weighted spanning tree enumerator splits into linear factors.
Another proof of their formula can be obtained from the foregoing formulas for threshold
graphs, as we now explain.  The key idea is due to Richard Ehrenborg.

Let $G$ be a connected threshold graph on vertices~$[n]$, and let~$m$ be the side length of the Durfee square of $G$.  Then the vertices $1,2,\dots,m$ are pairwise adjacent,
while $m+1,\dots,n$ are pairwise nonadjacent.  Moreover,
if $m+1\leq i<j\leq n$, then every neighbor of $j$ is a neighbor of $i$.
Construct a graph $F$ by deleting all edges $ij$ such that $i,j\leq m$.  Then $F$ is a
Ferrers graph; furthermore, all Ferrers graphs can be constructed in this way.
Thus, if we begin with the weighted enumerator for $G$ and
set to zero all indeterminates corresponding to edges between vertices
$1,2,\dots,m$, we recover the weighted
enumerator for the corresponding Ferrers graph $F$.  Specifically,
Theorem~\ref{shifted-theorem} yields
  \begin{equation} \label{fine-threshold}
  \hat \tau(G) = X_{1,1}X_{n,2}  \prod_{i=2}^{n-1}
    \left(\sum_{r=1}^{d_i^T(G)} X_{\min(i,r),1} X_{\max(i,r),2}\right)
  \end{equation}
(this is also \cite[Theorem~4]{MR2}).
Breaking up the product in \eqref{fine-threshold} around the parameter $m$ gives
  \begin{multline*}
  \tau(G) = X_{1,1}X_{n,2} \left(\prod_{i=2}^{m}
  \left( \sum_{r=1}^{m}  X_{\min(i,r),1} X_{\max(i,r),2} +
          \sum_{r = m+1}^{d_i^T(G)}  X_{\min(i,r),1} X_{\max(i,r),2} \right)\right) \x\\
  \left( \prod_{i=m+1}^{n-1} \sum_{r=1}^{d_i^T(G)}  X_{\min(i,r),1} X_{\max(i,r),2}\right).
  \end{multline*}

This expression is well defined because $d_i^T(G) \geq m$ whenever $i\leq m$.
If $i \leq m$ and $r \leq m$, then $\max(i,r)\leq m$, and these are exactly
the terms we wish to set to zero.  Therefore,
  \begin{equation} \label{fine-ferrers}
  \tau(F) = X_{1,1}X_{n,2}
  \left(\prod_{i=2}^{m} \sum_{r = m+1}^{d_i^T(G)}  X_{i,1} X_{r,2}\right)
  \left(\prod_{i=m+1}^{n-1} \sum_{r=1}^{d_i^T(G)}  X_{\min(i,r),1} X_{\max(i,r),2}\right).
  \end{equation}

For $i \geq m+1$, we have $d_i^T(G) < m < i$.  Thus $r<i$ for $r \leq d_i^T(G)$,
and \eqref{fine-ferrers} yields
  \begin{align}
  \tau(F) &= X_{1,1} X_{n,2}
    \left( \prod_{i=2}^{m} X_{i,1} \prod_{i=2}^{m} \sum_{r=m+1}^{d_i^T(G)} X_{r,2}\right)
    \left( \prod_{i=m+1}^{n-1} X_{i,2} \prod_{i=m+1}^{n-1} \sum_{r=1}^{d_i^T(G)}  X_{r,1} \right)\notag\\
  &= (X_{1,1}X_{2,1} \ldots X_{m,1})( X_{m+1,2} X_{m+2,2} \ldots X_{n,2})
    \left(\prod_{i=2}^{m} \sum_{r=m+1}^{d_i^T(G)} X_{r,2}\right)
    \left(\prod_{i=m+1}^{n-1} \sum_{r=1}^{d_i^T(G)}  X_{r,1}\right).
  \label{final-ferrers}
  \end{align}

By construction, the vertex degrees in $G$ and $F$ are related by the formula
  $$\deg_F(i) = \begin{cases}
    \deg_G(i)-m & \text{ if } 1\leq i\leq m,\\
    \deg_G(i) & \text{ if } m+1\leq i\leq n.
  \end{cases}$$

To simplify the notation, set $X_{r,1}=x_r$ and $y_{r-m}=X_{r,2}=y_{r-m}$.  (From this perspective, the two partite sets of $F$ correspond to the indeterminates
$\{x_1, \ldots x_m\}$ and $\{y_1, \ldots y_{n-m}\}$.  Therefore,
$$\tau(F) = (x_1 \ldots x_m)( y_1 \ldots y_{n-m})
  \left( \prod_{i=2}^{m} \sum_{r=1}^{d_i^T(F_2)} y_r\right)
  \left( \prod_{i=2}^{n-m} \sum_{r=1}^{d_i^T(F_1)} x_r\right)
$$
which is Theorem~2.1 of \cite{EvW}.


\end{document}